\documentclass[a4paper,11pt,reqno]{amsart} 
\usepackage{amssymb, amsmath, amsthm, mathtools, mathrsfs, bm}

\usepackage{tikz-cd}
\usepackage{extarrows}
\usepackage{arydshln} 
\usepackage{enumerate}
\usepackage{booktabs,float}
\usepackage{bbm}
\usepackage{xfrac}
\usepackage[square, comma, sort, numbers]{natbib}
\usepackage[colorlinks = true,
            linkcolor = blue,
            urlcolor  = cyan,
            citecolor = black,
            anchorcolor = black]{hyperref}

\theoremstyle{plain}
\newtheorem{theorem}{Theorem}[section]
\newtheorem{proposition}[theorem]{Proposition}
\newtheorem{corollary}[theorem]{Corollary}
\newtheorem{lemma}[theorem]{Lemma}

\theoremstyle{definition}

\newtheorem{example}[theorem]{Example}

\theoremstyle{remark}
\newtheorem{remark}[theorem]{Remark}

\newcommand{\xra}{\xrightarrow}

\newcommand{\PP}{\ensuremath{\mathcal{P}^1}}
\newcommand{\PD}{Poincar\'{e} duality\,}

\newcommand{\loca}{\ensuremath{(\frac{1}{2})}}
\newcommand{\Z}{\ensuremath{\mathbb{Z}}}

\newcommand{\z}[1]{\ensuremath{\mathbb{Z}/2^{#1}}}
\newcommand{\zp}[1]{\ensuremath{\mathbb{Z}/p^{#1}}}

\DeclareMathOperator{\coker}{coker}
\DeclareMathOperator{\im}{im}
\DeclareMathOperator{\Sq}{Sq}

\newcommand{\matwo}[2]{\ensuremath{\footnotesize{\begin{bmatrix}
  #1\\
  #2
\end{bmatrix}}}}

\newcommand{\smatwo}[2]{\ensuremath{\left[\begin{smallmatrix}
 #1\\
 #2
\end{smallmatrix}\right]}}

\title[Suspension of $(2n+2)$-dimensional Poincar\'{e} complexes]{Suspension Homotopy of $(n-1)$-connected $(2n+2)$-dimensional Poincar\'{e} duality complexes}

\author[P. Li]{Pengcheng Li}
\address{Department of Mathematics, School of Sciences, Great Bay University, Dongguan, Guangdong \rm{523000}, China}
\email{lipcaty@outlook.com}
\urladdr{https://lipacty.github.io}

\author[Z. Zhu]{Zhongjian Zhu}
\address{College of Mathematics and Physics, Wenzhou University, Wenzhou, Zhejiang \rm{325035}, China}
\email{zhuzhongjian@amss.ac.cn}


\begin{document}

\begin{abstract}
We study the homotopy decompositions of the suspension $\Sigma M$ of an $(n-1)$-connected $(2n+2)$ dimensional Poincar\'{e} duality complex $M$, $n\geq 2$.  
In particular, we completely determine the homotopy types of $\Sigma M$ of a simply-connected orientable closed (smooth) $6$-manifold $M$, whose integral homology groups can have $2$-torsion. If $3\leq n\leq 5$, we obtain homotopy decompositions of $\Sigma M$ after localization away from $2$.
\end{abstract}
\maketitle

\tableofcontents

\section{Introduction}\label{sect:intro}


In \cite{ST19} So and Theriault found that the homotopy decomposition of the (double) suspension space of manifolds have directly or implicitly applications to important concepts in geometry and physics, such as topological $K$-theory, gauge groups and current groups.
Since them, the topic of suspension homotopy of manifolds becomes popular, such as \cite{Huang21,Huang22, Huang-arxiv, CS22,HL,lipc23}. As can be seen in these papers, the $2$-torsion of the homology groups of the given manifolds usually causes great obstruction to obtain a complete characterization of the homotopy decomposition of the (double) suspension. 

This paper contributes to the suspension homotopy type of \PD complexes, which are connected CW-complexes whose integral cohomology satisfying the \PD theorem.
Let $M$ be an $(n-1)$-connected $(2n+2)$ dimensional \PD complex, $n\geq 2$. 
By \PD,  the integral homology groups $H_\ast(M)$ is given by 
\begin{equation}\label{HM}
  \begin{tabular}{cccccccc}
    \toprule
$i$& $n$&$n+1$&$n+2$&$0,2n+2$&$\text{otherwise}$
  \\\midrule
  $H_i(M)$& $\Z^l\oplus T$&  $\Z^d\oplus T$ & $\Z^l$& $\Z$&$0$
  \\ \bottomrule
  \end{tabular},
\end{equation}
where $T= T_2\oplus T_{\neq 2}$ and $T_2= \oplus_{i=1}^{t_2}\z{r_i}$ is the $2$-primary component.
When $M$ is a simply-connected orientable closed (smooth) $6$-manifold, Huang \cite{Huang-arxiv} firstly determined the homotopy type of the double suspension $\Sigma^2 M$ under the assumption that $T$ contains no $2$- or $3$-torsions, while Cutler and So \cite{CS22} obtained the homotopy decompositions of $\Sigma M$ after localization away from $2$. 
Our first main theorem (Theorem \ref{thm:6mfds} below) gives a complete characterization of the homotopy types of $\Sigma M$, for which the $2$-torsion $T_2$ can be non-zero. 
The new and key point is that we use  the second James-Hopf invariant to prove that the attaching map of the top cell of the given $6$-manifold is a suspension map and contains no Whitehead products, see Proposition \ref{prop:n=2:key}.

To state our main results, we need the following \emph{stable} cohomology operations. 
For each prime $p$ and integer $r\geq 1$, there are the higher order Bocksteins $\beta_r$ which detects the degree $2^r$ map on spheres $S^n$. 
For each $n\geq 2$, there are secondary cohomology operations
\begin{equation}\label{eq:Theta}
  \Theta_n\colon S_n(X)\to T_n(X)
 \end{equation}  
 detecting the map $\eta^2\in\pi_{n+2}(S^n)$, where 
 \begin{align*}
  S_n(X)&=\ker(\theta_n)_\sharp=\ker(\Sq^2)\cap \ker(\Sq^2\Sq^1)\\
  T_n(X)&=\coker(\Omega\varphi_n)_\sharp=H^{n+3}(X;\z{})/\im(\Sq^1+\Sq^2). 
\end{align*} 
That is, $\Theta_n$ is based on the null-homotopy of the composition
\[K_n\xra{\theta_n=\smatwo{\Sq^2\Sq^1}{\Sq^2}} K_{n+3}\times K_{n+2}\xra{\varphi_n=[\Sq^1,\Sq^2]}K_{n+4},\]
where $K_m=K(\z{},m)$ denotes the Eilenberg-MacLane space of type $(\z{},m)$.
The map $\eta^3$ can be detected by the tertiary operation $\mathbb{T}$, see \cite[Exercise 4.2.5]{Harperbook}. 
For $m\geq 3$, let $\alpha_1(m)=\Sigma^{m-3}\alpha_1(3)\colon S^{m+3}\to S^m$
be the generator of the $3$-primary component of $\pi_{m+3}(S^m)$. Denote by $\tilde{\alpha}_1(m)$ the possible  lift (if exists) of $\alpha_1(m)$ to $P^m(3^r)$; i.e., $q_m\circ \tilde{\alpha}_1(m)=\alpha_1(m)$.
See \cite[Lemma 2.6]{LPW} or \cite[Chapter XIII]{Todabook} with $p=3$. Note that $\alpha_1(m)$ and $\tilde{\alpha}_1(m)$ can be detected by the first reduced Steenrod power operation (cf. \cite[Chapter 1.5.5]{Harperbook}).
\begin{equation}\label{PP}
  \PP\colon H^\ast(-;\Z/3)\to H^{\ast+4}(-;\Z/3).
\end{equation}

We also need the following notations. For a finitely generated group $G$ and an integer $n\geq 2$, denote by $P^n(G)$ the $n$-dimensional Peterson space which is characterized by its unique non-trivial reduced integral cohomology group $G$ in dimension $n$. For $n,k\geq 2$, denote by $P^n(k)=P^n(\Z/k)$
the Moore space of homotopy type $(\Z/k,n-1)$ and $\Z/k$ is the group of integers mod $k$. There is a canonical homotopy cofibration
\[S^{n-1}\xra{k}S^{n-1}\xra{i_{n-1}}P^n(k)\xra{q_n}S^m,\]
where $t$ is the degree $t$ map, $i_{n-1}$ and $q_n$ are the canonical inclusion and pinch maps, respectively. 
We shall also use the following \emph{elementary Chang-complexes} (due to Chang \cite{Chang50}) with $n\geq 2,r\geq 1$:
\begin{align*}
  C^{n+2}_\eta&=S^n\cup_\eta e^{n+2},\quad  C^{n+2}_r=P^{n+1}(2^r)\cup_{i_{n}\eta}e^{n+2},
\end{align*}
which admit the folloing homotopy cofibrations (cf. \cite{ZP17})
\begin{align*}
 & S^{n+1}\xra{\eta}S^n\xra{i_n}C^{n+2}_\eta\xra{q_{n+2}}S^6;\\
 &S^{n+1}\vee S^{n}\xra{(\eta,2^r)}S^{n}\xra{i_n}C^{n+2}_r\xra{q}S^{n+2}\vee S^{n+1},\\
 &S^{n+1}\xra{i_n\eta}P^{n+1}(2^r)\xra{i_P}C^{n+2}_r\xra{q_{n+2}}S^{n+2},\\
 &S^{n}\xra{i_n2^r}C^{n+2}_\eta\xra{i_\eta}C^{n+2}_r\xra{q_{n+1}}S^{n+1},
\end{align*}
where $i_m\colon S^m\to Y$ and $q_m\colon X\to S^m$ respectively denote the appropriate inclusion and pinch maps for different complexes $X,Y$; $i_P,i_\eta$ are the canonical inclusion into $C^{n+2}_r$. 

Now it is prepared to state the first main theorem.

\begin{theorem}\label{thm:6mfds}
  Let $M$ be a simply-connected closed $6$-manifold with $H_\ast(M)$ given by (\ref{HM}), $n=2$. Let $\omega_2(M)$ be the second Stiefel-Whitney class of $M$. Let $T[c_2]=T\big/\oplus_{j=1}^{c_2}\z{r_j}.$
  There exist integers $0\leq c_1\leq l,0\leq c_2\leq t_2$ such that $0\leq c_1+c_2\leq l$, and $c_1=c_2=0$ if and only if the Steenrod square $\Sq^2$ acts trivially on $H^2(M;\z{})$. 
  \begin{enumerate}[1.]
  \item\label{thm-spin}  Suppose that $\omega_2(M)=0$, then the possible homotopy types of $\Sigma M$ can be distinguished by the the secondary operation $\Theta$: 
  \begin{enumerate}
    \item\label{spin-Theta=0} If $\Theta$ acts trivially on $H^3(M;\z{})$, then there is a homotopy equivalence 
    \begin{multline*}
      \Sigma M\simeq \big(\bigvee_{i=1}^dS^4\big)\vee P^5(T)\vee \big(\bigvee_{i=2}^{l-c_1} S^{3}\big) 
      \vee \big(\bigvee_{i=1}^{l-c_1-c_2} S^{5}\big)\vee\big(\bigvee_{i=1}^{c_1} C^{5}_\eta\big)\\\vee P^4(T[c_2])\vee \big(\bigvee_{j=1}^{c_2}  C^{5}_{r_j}\big) \vee (S^3\cup_{\epsilon\cdot\eta^3}e^7),
     \end{multline*}
    where  $\epsilon\in\{0,1\}$, and $\epsilon=0$ if the tertiary operation $\mathbb{T}$ acts trivially on $H^2(M;\z{})$ or $l=c_1=1$.
    \item\label{spin-Thetaneq0-1} If for any $u,v\in H^4(\Sigma M;\z{})$ with $\Theta(u)\neq 0$, $\Theta(v)=0$, there holds $u+v\notin \im(\beta_s)\text{ for any } s\geq 1;$
    while there exist $u',v'\in H^4(\Sigma M;\z{})$ such that 
   $\Theta(u')\neq 0$, $\Theta(v')=0$ and $\beta_r(u'+v')\neq 0,$ then there is a homotopy equivalence 
    \begin{multline*}
      \Sigma M\simeq \big(\bigvee_{i=1}^dS^4\big)\vee \big(\bigvee_{i=1}^{l-c_1} S^{3}\big) \vee \big(\bigvee_{i=1}^{l-c_1-c_2} S^{5}\big)
      \vee P^4(T[c_2])\vee \big(\bigvee_{j=1}^{t_2}  C^{5}_{r_j}\big)\\ \vee\big(\bigvee_{i=1}^{c_1} C^{5}_\eta\big)\vee P^5\big(\frac{T}{\z{r_{j_1}}}\big)\vee (P^5(2^{r_{j_1}})\cup_{i_4\eta^2}e^7),
     \end{multline*} 
where $r_{j_1}$ is the maximum of $r_{j}$ such that $\beta_{r_j}(u'+v')\neq 0$. 
    \item \label{spin-Thetaneq0-2} If there exist $u\in H^4(\Sigma M;\z{})$ with $\Theta(u)\neq 0$ and $v\in \ker(\Theta)$ such that $u+v\in\im(\beta_r)$ for some $r$,
    then there is a homotopy equivalence 
    \begin{multline*}
      \Sigma M\simeq \big(\bigvee_{i=1}^dS^4\big)\vee P^5(T)\vee \big(\bigvee_{i=1}^{l-c_1} S^{3}\big) 
      \vee \big(\bigvee_{i=1}^{l-c_1-c_2} S^{5}\big)\vee\big(\bigvee_{i=1}^{c_1} C^{5}_\eta\big)\vee \big(\bigvee_{j=1}^{c_2}  C^{5}_{r_j}\big)\\\vee P^4\big(\frac{T[c_2]}{\z{r_{j_2}}}\big) \vee (P^4(2^{r_{j_2}})\cup_{\tilde{\eta}_{r_{j_2}}\eta}e^7), \text{ or}\\[1ex]
      \Sigma M\simeq \big(\bigvee_{i=1}^dS^4\big)\vee P^5(T)\vee \big(\bigvee_{i=1}^{l-c_1} S^{3}\big) \vee \big(\bigvee_{i=1}^{l-c_1-c_2} S^{5}\big)
      \vee\big(\bigvee_{i=1}^{c_1} C^{5}_\eta\big)\\\vee P^4(T[c_2])\vee \big(\bigvee_{j_1\neq j=1}^{c_2}  C^{5}_{r_j}\big) \vee (C^5_{r_{j_2}}\cup_{i_P\tilde{\eta}_{r_{j_2}}\eta}e^7),
     \end{multline*} 
where $r_{j_2}$ is the minimum of $r_j$ such that $u+v\in \im(\beta_{r_i})$. The second possibility doesn't exist if $l=c_1=1$.
     
   \end{enumerate}
   \item\label{thm-Sq2eq0} Suppose that $\omega_2(M)\neq 0$, then $\Sigma M$ has the following possible homotopy decompositions:
   \begin{enumerate}
     \item\label{Sq2eq0-Sq2neq0-1} If for any $u\in H^5 (\Sigma M;\z{})$ with $\Sq^2(u)\neq 0$ and any $v\in \ker(Sq^2)$, there hold $u+v\notin \im (\beta_r) \text{ for any }~r\geq 1,$
   then there is a homotopy equivalence 
   \begin{multline*}
    \Sigma M\simeq \big(\bigvee_{i=1}^dS^4\big)\vee  P^5(T)\vee \big(\bigvee_{i=1}^{l-c_1} S^{3}\big)\vee \big(\bigvee_{i=2}^{l-c_1-c_2} S^{5}\big)
     \vee\big(\bigvee_{i=1}^{c_1} C^{5}_\eta\big)\\\vee P^4(T[c_2])\vee \big(\bigvee_{j=1}^{c_2}  C^{5}_{r_j}\big)  \vee C^7_\eta .
   \end{multline*}
     \item\label{Sq2eq0-Sq2neq0-2} If there exist $u\in H^5(\Sigma M;\z{})$ with $\Sq^2(u)\neq 0$, $v\in \ker(\Sq^2)$ such that $u+v\in \im (\beta_{r_i}) $ for some $r_i$,
     then there is a homotopy equivalence 
     \begin{multline*}
      \small \Sigma M\simeq \big(\bigvee_{i=1}^dS^4\big)\vee P^5\big(\frac{T}{\z{r_{i_1}}}\big)\vee \big(\bigvee_{i=1}^{l-c_1} S^{3}\big) \vee \big(\bigvee_{i=1}^{l-c_1-c_2} S^{5}\big)\vee\big(\bigvee_{i=1}^{c_1} C^{5}_\eta\big)
       \\\small \vee P^4(T[c_2])\vee \big(\bigvee_{j=1}^{c_2}  C^{5}_{r_j}\big) \vee \big(P^5(2^{r_{i_1}})\cup_{\tilde{\eta}_{r_{i_1}}}e^7)\big),
     \end{multline*}
     where $r_{i_1}$ is the minimum of $r_i$ such that $u+v\in \im(\beta_{r_i})$.
   \end{enumerate}
  \end{enumerate}
\end{theorem}

\begin{remark}\label{rmk:6-mflds} Although we list all the possibilities of homotopy types of $\Sigma M$ of the given $6$-manifolds, we cannot guarantee that all possibilities could happen and one could find a $6$-manifold as an illustrating example. This is partially due to our lack of knowledge about manifolds. On Theorem 1.1, we also make the following remarks.
\begin{enumerate}
  \item $\omega_2(M)=v_2(M)$ (the second Wu class \cite[page 132]{MS74}) is a homotopy invariant for orientable topological $6$-manifolds and does not depend on the smooth structure on $M$. When $M$ is smooth, $\omega_2(M)=0$ if and only if $M$ is spin.
  \item  If $l=1$, then $(c_1,c_2)\in \{(0,0),(0,1),(1,0)\}$ can be determined by the action of the Steenrod square on $H^2(M;\z{})$, see Proposition \ref{prop:2n+1} (3). 
  \item  Convention: Wedge summands like $\bigvee_{i=k}^{m}X_i$ with $k>m$ should be viewed as a point and can be removed. 
\end{enumerate}
\end{remark}  

Denote by $X\simeq_{(\frac{1}{2})}Y$ if the spaces $X$ and $Y$ have the same homotopy type after localization away from $2$. The following immediate corollary is an improvement of \cite[Theorem 1.1]{CS22}.
\begin{corollary}\label{cor:6mfds}
  Let $M$ be a simply-connected $6$-manifold as in Theorem \ref{thm:6mfds}.  Then there is a homotopy equivalence 
  \[\Sigma M\simeq_{(\frac{1}{2})}\big(\bigvee_{i=1}^dS^4\big)\vee \big(\bigvee_{i=1}^l(S^3\vee S^5)\big)\vee P^4(T)\vee P^5(T)\vee S^7.\]
\end{corollary}

For $n=3,4,5$, we have the following results on the homotopy types of the suspension of manifolds after localization away from $2$.

\begin{theorem}\label{thm:8-mflds}
  Let $M$ be a $2$-connected $8$ dimensional \PD complex with $H_\ast(M)$ given by (\ref{HM}), $n=3$.
  \begin{enumerate}[1.]
    \item Suppose $\PP\big(H^4(M;\Z/3)\big)=0$, then there is a homotopy equivalence 
\[\Sigma M\simeq_{\loca} \big(\bigvee_{i=1}^l(S^4\vee S^6)\big)\vee \big(\bigvee_{i=1}^dS^5\big)\vee P^5(T)\vee P^6(T)\vee S^{9}. \]
\item Suppose $\PP\big(H^4(M;\Z/3)\big)\neq 0$, then the homotopy type of $\Sigma M$ can be distinguished by the higher order Bockstein $\beta_r$: 
\begin{enumerate}
  \item If for any $u\in H^5(\Sigma M;\Z/3)$ with $\PP(u)\neq 0$ and any $v\in \ker(\PP)$, there hold $\beta_r(u+v)=0,\quad u+v\notin\im(\beta_s)$ for any $r,s\geq 1, $
  then there is a homotopy equivalence 
  \begin{multline*}
   \small \Sigma M\simeq_{\loca}\big(\bigvee_{i=1}^l(S^4\vee S^6)\big)\vee (\bigvee_{i=2}^dS^5)\vee P^5(T)\vee P^6(T)\vee (S^5\cup_{\alpha_1(5)}e^9).
  \end{multline*}

  \item If there exist $u\in H^5(\Sigma M;\Z/3)$ with $\PP(u)\neq 0$ and $v\in \ker(\PP)$ such that $u+v\in \im(\beta_{r})~\text{ for some $r$},$
  then there is a homotopy equivalence 
  \begin{multline*}
  \small  \Sigma M\!\simeq_{\loca}\!\big(\bigvee_{i=1}^l(S^4\vee S^6)\big)\!\vee\! (\bigvee_{i=2}^dS^5)\!\vee\! P^5\big(\frac{T}{\Z/3^{r_{j_1}}}\big)\!\vee\! P^6(T)
  \!\vee\!\big(P^5(3^{r_{j_1}})\cup_{\tilde{\alpha}_1(5)}e^9\big),
  \end{multline*}
  where $r_{j_1}$ is the minimum of $r_j$ such that $u+v\in \im(\beta_{r_j})$.

  \item If for any $u\in H^5(\Sigma M;\Z/3)$ with $\PP(u)\neq 0$ and any $v\in \ker(\PP)$, there hold $u+v\notin \im(\beta_s)$ for any $s\geq 1,$
   while there exist $u'\in H^5(\Sigma M;\Z/3)$ with $\PP(u')\neq 0$ and $v'\in \ker(\PP)$ such that $\beta_r(u'+v')\neq 0~\text{ for some $r\geq 1$},$
then there is a homotopy equivalence 
\begin{multline*}
  \small \Sigma M\!\simeq_{\loca}\!\big(\bigvee_{i=1}^l(S^4\vee S^6)\big)\!\vee\! (\bigvee_{i=2}^dS^5)\!\vee\! P^5(T)\!\vee\! P^6\big(\frac{T}{\Z/3^{r_{j_2}}}\big)
  \!\vee\!\big(P^6(3^{r_{j_2}})\cup_{i_5\alpha_1(5)}e^9\big),
\end{multline*}
where $r_{j_2}$ is the maximum of $r_j$ such that $\beta_{r_j}(u'+v')\neq 0$.  
\end{enumerate}

  \end{enumerate}
\end{theorem}

\begin{theorem}\label{thm:p-odd:n=4}
  Let $M$ be a $3$-connected $10$ dimensional \PD complex with $H_\ast(M)$ given by (\ref{HM}), $n=4$. If $\PP$ acts trivially on $H^6(M;\Z/3)$, then there is a homotopy equivalence 
  \[\Sigma M\simeq_{\loca}\big(\bigvee_{i=1}^l(S^5\vee S^7)\big) \vee \big(\bigvee_{i=1}^dS^6\big)\vee P^6(T)\vee P^7(T)\vee S^{11};\]
  otherwise there is a homotopy equivalence 
 \begin{multline*}
 \small \Sigma M\!\simeq_{\loca}\!\big(\bigvee_{i=1}^l(S^5\vee S^7)\big) \!\vee\! \big(\bigvee_{i=1}^dS^6\big)\!\vee\! P^6(T)\!\vee\! P^7(\frac{T}{\Z/3^{r_{j_1}}})\!\vee\! \big(P^7(3^{r_{j_1}})\cup_{\tilde{\alpha}_1(7)} e^{11}\big),
 \end{multline*}
  where $r_{j_1}$ is the minimum of $r_j$ such that  
 $\PP(u+v)\neq 0, u+v\in\im(\beta_{r_j})$
  for some $u\in H^7(\Sigma M;\Z/3)$ and $v\in \ker(\PP)$.

\end{theorem}

\begin{theorem}\label{thm:p-odd:n=5}
  Let $M$ be a $4$-connected $12$ dimensional \PD complex  with $H_\ast(M)$ given by (\ref{HM}), $n=5$. There is a homotopy equivalence 
  \[\Sigma M\simeq_{\loca} \big(\bigvee_{i=1}^l(S^6\vee S^8)\big)\vee \big(\bigvee_{i=1}^dS^7\big)\vee P^7(T)\vee P^8(T)\vee S^{13}.\]
\end{theorem}

The basis method to study the homotopy type of $\Sigma M$ is the homology decomposition. In the cases $2\leq n\leq 5$, we only need to solve the obstruction caused by the $2$- and $3$-torsions of $H_\ast(M)$. For $n=4,5$, the results and discussions are similar to but easier than that in the case $n=3$, so we omit their proofs here. 
If $n=6$, there hold isomorphisms
  \[\pi_{13}(P^7(p^r))\cong\pi_{14}(P^8(p^r))\cong\Z/(15,p^r).\]
 It follows that in general the reduced Steenrod $3$-th power $\PP$, together with the Bocksteins $\beta_r$, is not enough to characterize the homotopy decomposition of the suspension of a $(n-1)$-connected $(2n+2)$ dimensional \PD complex for $n\geq 6$.  

The paper is arranged as follows. In Section \ref{sect:prelim} we compute some homotopy groups of some elementary complexes and list some lemmas on cohomology operations.  In Section \ref{sect:homldecomp} we determine the homotopy type of the $(2n+1)$-th homology section of the \PD complex $M$. Section \ref{sect:n=2} and Section \ref{sect:n=3} cover the proof of Theorem \ref{thm:6mfds} and Theorem \ref{thm:8-mflds}, respectively. In Section \ref{sect:JH} we compute the homomorphisms induced by the second James-Hopf invariant on the $6$-th homotopy groups of the wedge sum of some elementary complexes, which play essential role in Section \ref{sect:n=2}.

\section*{Acknowledgements}
Pengcheng Li was supported by National Natural Science Foundation of China (Grant No. 12101290), Zhongjian Zhu was supported by National Natural Science Foundation of China
(Grant No. 11701430).

\section{Preliminaries}\label{sect:prelim}
Throughout the paper all spaces are based CW-complexes, all maps are base-point-preserving and are identified with their homotopy classes in notation. 
For an abelian group $G$ generated by $x_1,\cdots,x_n$, we use the notation
\[G\cong C_1\langle x_1\rangle\oplus\cdots\oplus C_n\langle x_n\rangle \text{ or }G\cong C_1\oplus\cdots\oplus C_n\langle x_1,\cdots,x_n\rangle \]  
to indicate $x_i$ is a generator of the cyclic direct summand $C_i$, $i=1,\cdots,n$. For a prime $p$, denote by $\pi_i(X;p)$ the $p$-primary component of the $i$-th homotopy group $\pi_i(X)$.

\subsection{Some homotopy groups of elementary complexes}
 We will frequently use the followings (cf. \cite{Todabook}): 
\begin{enumerate}[\ ]
  \item $\pi_3(S^2)\cong\Z\langle \eta\rangle$,  $\pi_{m+1}(S^m)\cong \z{}\langle \eta \rangle$ for $m\geq 3$, $\pi_{m+2}(S^m)\cong \z{}\langle \eta^2 \rangle$;
  \item $\pi_6(S^3)\cong\Z/12\langle \nu' \rangle$, $\pi_6(S^3;3)\cong\Z/3\langle \alpha_1(3)\rangle$; $\pi_7(S^4)\cong \Z\langle \nu\rangle\oplus\Z/12\langle \Sigma \nu'\rangle$,  $\pi_{m+3}(S^m)\cong\Z/{24}\langle \nu_m\rangle$ for $m\geq 5$. Note $\eta^3=6\nu'$, $2\nu_m=\Sigma^{m-3}\nu'$ for $m\geq 5$. Denote by $\delta_1=1$ and $\delta_r=0$ for $r\geq 2$.
  
\end{enumerate}

\begin{lemma}[cf. \cite{BH91}]\label{lem:Moore1}
    Let $n\geq 3,r\geq 1$ be integers and let $p$ be a prime.  The followings hold:
   \begin{enumerate}[(1)]
    \itemsep=1ex
    \item $\pi_n(P^{n+1}(p^r))\cong\zp{r}\langle i_n\rangle$.
    \item $\pi_{n+1}(P^{n+1}(p^r))\cong \left\{\begin{array}{ll}
  \z{}\langle i_n\eta\rangle,&p=2;\\
  0,&p\geq 3.
    \end{array}\right.$

    \item $\pi_{n+2}(P^{n+1}(p^r))\cong \left\{\begin{array}{ll}
      \Z/4\langle \tilde{\eta}_1\rangle,&p=2,~r=1;\\
      \Z/2\oplus\z{}\langle \tilde{\eta}_r,i_n\eta^2\rangle,&p=2,~r\geq 2;\\
      0,&p\geq 3. 
    \end{array}\right.$
    \\
  The generator $\tilde{\eta}_r$ satisfies formulas
  \begin{equation}\label{eq:eta}
    q_{n+1}\tilde{\eta}_r=\eta,\quad 2\tilde{\eta}_1=i_n\eta^2.
  \end{equation}
  \item $[P^{n+1}(2^r),P^{n+1}(2^s)]\cong \left\{\begin{array}{ll}
    \Z/4\langle 1_P\rangle,&r=s=1;\\
    \z{m}\oplus\z{}\langle B(\chi^r_s),i_n\eta q_{n+1}\rangle,&\text{$r$ or $s>1$},
  \end{array}\right.$ \\
where $m=\min(r,s)$, $B(\chi^r_s)$ satisfies $\Sigma B(\chi^r_s)=B(\chi^r_s)$ and 
\begin{equation}\label{eq:chi}
  B(\chi^r_s)i_{n}=\left\{\begin{array}{ll}
        i_{n},&r\geq s;\\
        2^{s-r}i_{n},&r\leq s.
      \end{array}\right. ~~q_{n+1}B(\chi^r_s)=\left\{\begin{array}{ll}
        2^{r-s}q_{n+1},&r\geq s;\\
        q_{n+1},&r\leq s.
      \end{array}\right.
\end{equation}
\begin{equation}\label{eq:chi-eta}
  B(\chi^r_s)\tilde{\eta}_r=\tilde{\eta}_s \text{ for }s\geq r.
\end{equation}
   \end{enumerate}
  
\end{lemma}

The homotopy fibration $F^4\{2^r\}\xra{j_r}P^4(2^r)\xra{q_4} S^4$
indicates the inclusion $i_3\colon S^3\to P^4(2^r)$ factors as $S^3\xra{\imath_r} F^4\{2^r\}\xra{j_r}P^4(2^r)$.
 From \cite[Section 3.1]{ZP23}, there is a canonical map $\lambda_r\colon S^6\to F^4\{2^r\}$,
which is the skeletal inclusion map up to a homotopy equivalence.

\begin{lemma}[cf. \cite{ZP23}]\label{lem:Moore-n+3}
 Let $r\geq 1$. There hold isomorphisms 
 \begin{enumerate}
  \item $\pi_{6}(P^4(2^r))\cong \z{\varrho_r}\langle j_r\lambda_r\rangle\oplus\z{1-\delta_r} \langle i_3\nu' \rangle\oplus\z{}\langle \tilde{\eta}_r\eta \rangle$, where $\varrho_r=r+1$ for $r=1,2$, and $\varrho_r=r$ for $r\geq 3$.
  \item  $\pi_{7}(P^5(2^r))\cong \z{r+1}\oplus\z{\min(r-1,2)}\oplus\z{}\langle i_4\nu_4, i_4(\Sigma \nu'), \tilde{\eta}_r\eta\rangle$; it follows that $j_r\lambda_r$ in (1) is not a suspension element.
 \end{enumerate}
\end{lemma}

From \cite[Proposition 8.2 (1)]{Mukai82} and computations in \cite{ZP23} we have  
\begin{lemma}\label{lem:htpgrps-Chang-dim6}
  There hold isomorphisms
  \begin{enumerate}
    \item $\pi_6(C^5_\eta)\cong \Z/6\langle i_3\nu' \rangle$; $i_3\alpha_1(3)=2i_3\nu'$.
    \item $\pi_6(C^5_r)\cong \z{r+\delta_r}\langle i_Pj_r\lambda_r \rangle\oplus \z{1-\delta_r}\langle i_3\nu'\rangle\oplus\z{}\langle i_P\tilde{\eta}_r\eta\rangle$. 
  \end{enumerate}
\end{lemma}

The following Lemma can be found in \cite[Theorem 3.1, (2)]{lipc22}; since it hasn't been published yet, we give it a proof here.  

\begin{lemma}\label{lem:eq:C_r}
Let $n\geq 3$. There exists a map $\bar{\zeta}\colon C^{n+2}_\eta\to S^n$ satisfying
\begin{equation}\label{eq:zeta}
 \bar{\zeta}i_n=2\cdot 1_n, \quad 1_n=id_{S^n}.
\end{equation} 
If $r\leq s$, there exists a map $\bar{\xi}_r\colon C^{n+2}_r\to P^{n+1}(2^{r+1})$    
fitting the following homotopy commutative diagram 
  \[\begin{tikzcd}
 S^{n}\ar[r,"i_n2^r"]\ar[d,equal] & C^{n+2}_\eta \ar[r,"i_\eta"]\ar[d,"\bar{\zeta}"]& C^{n+2}_r\ar[d,"\bar{\xi}_r"]\ar[r,"q_{n+1}"]&S^{n+1}\ar[d,equal]\\
 S^{n}\ar[r,"2^{r+1}"]& S^n\ar[r,"i_n"]& P^{n+1}(2^{r+1})\ar[r,"q_{n+1}"]&S^{n+1}
\end{tikzcd}.\]
Moreover, there hold formulas  
\begin{equation}\label{eq:C_r-P}
  \bar{\xi}_r\circ i_P=B(\chi^r_{r+1}),\quad B(\chi^{s+1}_r)\bar{\xi}_s (i_P\tilde{\eta}_s\eta)\simeq \tilde{\eta}_r\eta \text{ for }r>s.
\end{equation}
It follows that for $r\leq s$, there exists a map $\alpha^r_s\colon C^{n+2}_r\to C^{n+2}_s$, which is the identity map for $r=s$ and is $i_PB(\chi^{r+1}_s)\bar{\xi}$ for $r<s$ such that 
\begin{equation}\label{eq:C_rs}
  r\leq s:~~ \alpha^r_s\circ i_P=i_P B(\chi^r_s),\quad \alpha^r_s(i_P\tilde{\eta}_r\eta)=i_P\tilde{\eta}_s\eta. 
\end{equation} 

\begin{proof}
(1) The relation \ref{eq:zeta} refers to \cite[(8.3)]{Toda2}.

(2) By (\ref{eq:zeta}) the first square in the Lemma is homotopy commutative, and hence the map $\bar{\xi}_r$ in the Lemma exists. Recall we have the composition 
\[i_n=i_\eta \circ i_n \colon S^n\to C^{n+2}_\eta\to C^{n+2}_r.\]
Then $\bar{\xi}_ri_n=(\bar{\xi}_ri_\eta) i_n=(i_n\bar{\zeta})i_n=2i_n$
imply that 
\[\bar{\xi}_r\circ i_P=B(\chi^r_{r+1})+\varepsilon\cdot i_n\eta q_{n+1}\] for some $\varepsilon \in\{0,1\}.$
If $\varepsilon=0$, we are done; otherwise we replace $\bar{\xi}_r$ by $\bar{\xi}_r+i_n\eta q_{n+1}$ to make $\varepsilon=0$. Note that all the relations mentioned above still hold even if we make such a replacement.  

The formula (\ref{eq:C_rs}) then follows by the relations (\ref{eq:C_r-P}), (\ref{eq:chi}) and (\ref{eq:chi-eta}).
\end{proof}
\end{lemma}

Let $p$ be an odd prime and let $S^{2n+1}\{p^r\}$ and $T^{2n+1}\{p^r\}$ be defined by the homotopy fibrations in which all spaces are $p$-local:
    \begin{align}
     \Omega S^{2n+1}\xra{} S^{2n+1}\{p^r\}\to S^{2n+1}\xra{p^r}S^{2n+1},\label{fib:S}\\
      W_n\times \varPi_{r+1}\to T^{2n+1}\{p^r\}\to \Omega S^{2n+1}\{p^r\},\label{fib:T}
    \end{align}
where $\varPi_{r+1}=\prod_{k=1}^{\infty}S^{2p^tn-1}\{p^{r+1}\}$,  $W_n$ is the homotopy fibre of the double suspension $E^2\colon S^{2n-1}\to \Omega^2S^{2n+1}$.
Recall that $S^{2n+1}\{p^r\}$ is $(2n-1)$-connected and $W_n$ is $(2pn-4)$-connected, we have 
\begin{equation}\label{eq:pi=0}
  \pi_{2n+1}(\varPi_{r+1})=\pi_{2n+1}(W_n)=0,~\forall~n\geq 3.
\end{equation}

\begin{lemma}\label{lem:p-odd:n=3}
  Let $p$ be an odd prime. There hold isomorphisms
  \[\pi_8(P^5(p^r))\cong\pi_8(P^6(p^r))\cong \Z/(3,p^r).\]
  For $p=3$, these two groups are respectively generated by $\tilde{\alpha}_1(5)$ and $i_5\alpha_1(5)$, where $q_5 \tilde{\alpha}_1(5)= \alpha_1(5)$ and $B(\chi^r_s)\tilde{\alpha}_1(5)=\tilde{\alpha}_1(5) $ for $s\geq r$.
  
  \begin{proof}
   By \cite{CMN79,CMN87} there are homotopy equivalences
    \begin{align}
      \Omega P^{2n+2}(p^r)&\simeq S^{2n+1}\{p^r\}\times\Omega \big(\bigvee_{k=0}^{\infty}P^{4n+2kn+3}(p^r)\big),\label{cmn-even}\\
      \Omega P^{2n+1}(p^r)&\simeq T^{2n+1}\{p^r\}\times \Omega \big(\bigvee_\alpha P^{n_\alpha}(p^r)\big),\label{cmn-odd}
  \end{align} 
    where $n_\alpha\geq 4n$ and for each $n$, there exists exactly one $\alpha$ such that $n_\alpha=4n$. 
It follows immediately that 
    \begin{align*}
      \pi_8(P^5(p^r))&\cong \pi_7(\Omega P^5(p^r))\cong \pi_7(T^5\{p^r\})\cong \pi_8(S^5\{p^r\}),\\
      \pi_8(P^6(p^r))&\cong\pi_7(\Omega P^6(p^r))\cong \pi_7(S^5\{p^r\}),
    \end{align*}
    where $\pi_8(S^5\{p^r\})\cong\Z/(3,p^r)$ and $\pi_7(S^5\{p^r\})\cong \Z/(3,p^r)$ can be easily computed by (\ref{fib:S}).

 By \cite[Proposition 2.3]{Neisen87}, there is a canonical map $\gamma_2\colon P^{5}(p^r)\to S^{5}\{p^r\}$ which is $8$-connected. It follows by the group structures that \[(\gamma_2)_\sharp\colon \pi_8(P^{5}(p^r))\xra{\cong} \pi_8(S^{5}\{p^r\}).\]
Hence there is a short exact sequence
\[0\to \pi_8(P^5(p^r))\xra{(q_5)_\sharp}\pi_8(S^5)\xra{p^r}\pi_8(S^5).\]
For $p=3$, we then get a natural isomorphism 
   \[\pi_8(P^5(3^r))\xra[\cong]{(q_5)_\sharp}\pi_8(S^5;3)\cong\Z/3\langle \alpha_1(5)\rangle.\]
  For $s\geq r$, by (\ref{eq:chi}) we have $B(\chi^r_s)\tilde{\alpha}_1(5)=\tilde{\alpha}_1(5)$.
  
 The group $\pi_8(P^6(p^r))$ is stable, its generator is clear.
  \end{proof}  
\end{lemma}

\subsection{Some lemma on cohomology operations}

 Consider the following complexes
\begin{align*}
 A^{n+3}(\tilde{\eta}_r)=P^{n+1}(2^r)\cup_{\tilde{\eta}_r}e^{n+3},&\quad  A^{n+3}(\eta^2)=S^n\cup_{\eta^2}e^{n+3},\\
   A^{n+3}({2^r}\eta^2)=P^{n+1}(2^r)\cup_{i_n\eta^2}e^{n+3},&\quad B^{n+3}(\tilde{\eta}_r\eta)=P^{n}(2^r)\cup_{\tilde{\eta}_r\eta}e^{n+3},\\
   C^{n+3}(i_P\tilde{\eta}_r\eta)=C^{n+1}_r\cup_{i_P\tilde{\eta}_r\eta}e^{n+3}.
\end{align*}

 \begin{lemma}[Section 2 of \cite{lipc23}]\label{lem:Sq2}
  For each $n\geq 2,r\geq 1$, the following Steenrod squares are isomorphisms:
  \begin{enumerate}
    \item\label{Sq2-Changcpx} $\Sq^2\colon H^n(X;\z{})\to H^{n+2}(X;\z{})$ with $X=C^{n+2}_\eta$ or $C^{n+2}_r$.
    \item\label{Sq2-eta_r} $\Sq^2\colon H^{n+1}(A^{n+3}(\tilde{\eta}_r);\z{})\to H^{n+3}(A^{n+3}(\tilde{\eta}_r);\z{}).$
  \end{enumerate}

 \end{lemma}

\begin{lemma}\label{lem:Theta}
Let  $n\geq 3,r\geq 1$ and let $X=A^{n+3}(\eta^2)$, $A^{n+3}(2^r\eta^2)$, $B^{n+3}(\tilde{\eta}_r\eta)$ or $C^{n+3}(i_P\tilde{\eta}_r\eta)$. Then 
\[0\neq \Theta\colon H^n(X;\z{})\to H^{n+3}(X;\z{}).\]
 
\begin{proof}
  For the complexes $X$ in the Lemma,  it is clear that 
    \[S_n(X)=H^n(X;\z{}),\quad T_n(X)=H^{n+3}(X;\z{}).\]
   $\Theta\neq 0$ for $X=A^{n+3}(\eta^2)$ or $A^{n+3}(2^r\eta^2)$ refers to \cite[Lemma 2.7]{lipc23}. 
 For $X=B^{n+3}(\tilde{\eta}_r\eta)$ or $C^{n+3}(i_P\tilde{\eta}_r\eta)$, There are homotopy commutative diagrams of homotopy cofibrations 
  \[\begin{tikzcd}[sep=scriptsize]
    \ast\ar[r]\ar[d]&S^{n-1}\ar[d,"i_{n-1}"]\ar[r,equal]&S^{n-1}\ar[d,"i_{n-1}"] \\
    S^{n+2}\ar[r,"\tilde{\eta}_r\eta"]\ar[d,equal]&P^{n}(2^r)\ar[r]\ar[d,"q_n"]& B^{n+3}(\tilde{\eta}_r\eta)\ar[d,"q"]\\
    S^{n+2}\ar[r,"\eta^2"]&S^n\ar[r]& A^{n+3}(\eta^2)
  \end{tikzcd},~\begin{tikzcd}[sep=scriptsize]
    \ast\ar[d]\ar[r]& S^n\ar[r,equal]\ar[d,"i_{n-1}\eta"]&S^n\ar[d,"i_{n-1}\eta"]\\
    S^{n+2}\ar[r,"\tilde{\eta}_r\eta"]\ar[d,equal]&P^n(2^r)\ar[d,"i_P"]\ar[r]&B^{n+3}(\tilde{\eta}_r\eta)\ar[d,"i"]\\
    S^{n+2}\ar[r,"i_P\tilde{\eta}_r\eta"]&C^{n+1}_r\ar[r]&C^{n+3}(i_P\tilde{\eta}_r\eta)
  \end{tikzcd}\]
 Then we get $\Theta\neq 0$ for $X=B^{n+3}(\tilde{\eta}_r\eta),C^{n+3}(i_P\tilde{\eta}_r\eta)$ from the following induced commutative diagram of mod $2$ cohomology groups (the coefficients $\z{}$ are omitted)
  \[\begin{tikzcd}[sep=scriptsize]
    H^n(A^{n+3}(\eta^2))\ar[r,"q^\ast","\cong"swap]\ar[d,"\Theta","\cong"swap]& H^n(B^{n+3}(\tilde{\eta}_r\eta))\ar[d,"\Theta"] &H^n(C^{n+3}(i_P\tilde{\eta}_r\eta))\ar[l,"i^\ast"swap,"\cong"]\ar[d,"\Theta"]\\
    H^{n+3}(A^{n+3}(\eta^2))\ar[r,"q^\ast","\cong"swap]& H^{n+3}(B^{n+3}(\tilde{\eta}_r\eta))&H^{n+3}(C^{n+3}(i_P\tilde{\eta}_r\eta))\ar[l,"i^\ast"swap,"\cong"]
  \end{tikzcd},\]
  where the isomorphisms $q^\ast$ and $i^\ast$ can be easily checked.
   \end{proof}
  \end{lemma}

  \begin{lemma}\label{lem:p-odd:n=3:P1}
  For $r,s\geq 1$, let $X=S^5\cup_{\alpha_1(5)}e^9$, $P^{6}(3^r)\cup_{i_5\alpha_1(5)}e^9$ or $X= P^{5}(3^s)\cup_{\tilde{\alpha}_1(5)}e^9$.
   The first Steenrod power $\PP$ acts non-trivially on $H^\ast(Y;\Z/3)$. 
   \begin{proof}
     It is well-known that $\PP$ is non-trivial for $Y=S^5\cup_{\alpha_1(5)}e^9$ (cf. \cite[Chapter 1.5.5]{Harperbook}).  For the latter two complexes, $\PP\neq 0$ can be easily deduced from the naturality of $\PP$ and the relations in Lemma \ref{lem:p-odd:n=3}.  
   \end{proof}
  \end{lemma}

\subsection{Matrix analysis method}

Let $X=\Sigma X'$, $Y_i=\Sigma Y_i'$ be suspensions, $i=1,2,\cdots,n$. Let \[i_l\colon Y_l\to \bigvee_{j=i}^nY_i,\quad p_k\colon \bigvee_{i=1}^n Y_i\to Y_k\] be respectively the canonical inclusions and projections, $1\leq k,l\leq n$. By the Hilton-Milnor theorem,  we may write a map
  \(f\colon X\to\bigvee_{i=1}^nY_i\)
  as \[f=\sum_{k=1}^n i_k\circ f_{k}+\theta,\]
  where $f_{k}=p_k\circ f\colon X\to Y_k$ and $\theta$ satisfies $\Sigma \theta=0$. The first part $\sum_{k=1}^n i_k\circ f_{k}$ is usually represented by a vector $u_f=(f_1,f_2,\cdots,f_n)^t.$ 
  We say that $f$ is completely determined by its components $f_k$ if $\theta=0$; in this case, denote $f=u_f$. Let $h=\sum_{k,l}i_lh_{lk}p_k$ be a self-map of $\bigvee_{i=1}^nY_i$ which is completely determined by its components $h_{kl}=p_k\circ h\circ i_l\colon Y_l\to Y_k$. Denote by   
  \[M_h\coloneqq (h_{kl})_{n\times n}=\begin{bmatrix}
    h_{11}&h_{12}&\cdots&h_{1n}\\
    h_{21}&h_{22}&\cdots&h_{2n}\\
    \vdots&\vdots&\ddots&\vdots\\
    h_{n1}&h_{n1}&\cdots&h_{nn}
  \end{bmatrix}\] 
  Then the composition law
  \(h(f+g)\simeq h f+h g\)
  implies that the product
  \[M_h[f_1,f_2,\cdots,f_n]^t\]
  given by the matrix multiplication represents the composite $h\circ f$.
  Two maps $f=u_f$ and $g=u_g$ are called \emph{equivalent}, denoted by 
  \[[f_1,f_2,\cdots,f_n]^t\sim [g_1,g_2,\cdots,g_n]^t,\]
  if there is a self-homotopy equivalence $h$ of $\bigvee_{i=1}^n Y_i$, which can be represented by the matrix $M_h$, such that 
  \[M_h[f_1,f_2,\cdots,f_n]^t\simeq [g_1,g_2,\cdots,g_n]^t.\] 
  Note that the above matrix multiplication refers to elementary row operations in matrix theory; and the homotopy cofibres of the maps $f=u_f$ and $g=u_g$ are homotopy equivalent if $f$ and $g$ are equivalent. 
 
  \begin{example}\label{ex}
 Using the relations (\ref{eq:C_r-P}), we have the equivalences for maps $S^6\to P^4(2^r)\vee C^5_s$: 
    \[\matwo{\tilde{\eta}_r\eta}{i_P\tilde{\eta}_s\eta}\sim  \matwo{\tilde{\eta}_r\eta}{0} \text{ if }r\leq s; \quad \matwo{\tilde{\eta}_r\eta}{i_P\tilde{\eta}_s\eta}\sim \matwo{0}{i_P\tilde{\eta}_s\eta}\text{ if }r>s.\] 
  
  \end{example}

  We end this section with the following useful lemma.

  \begin{lemma}[Lemma 6.4 of \cite{HL}]\label{lem:HL}
    Let $S\xra{f}(\bigvee_{i=1}^nA_i)\vee B \xra{g}\Sigma C$ be a homotopy cofibration of simply-connected CW-complexes. For each $j=1,\cdots, n$, let \[p_j\colon \big(\bigvee_{i}A_i\big)\vee B\to A_j,\quad q_B\colon \big(\bigvee_{i}A_i\big)\vee B\to B\] be the obvious projections. Suppose that the composite $p_jf$
  is null-homotopic for each $j\leq n$, then there is a homotopy equivalence
  \[\Sigma C\simeq \big(\bigvee_{i=1}^nA_i\big)\vee C_{q_Bf},\]
  where $C_{q_Bf}$ is the homotopy cofibre of the composite $q_Bf$.
  \end{lemma}

 \section{The $(2n+1)$-th homology section of $M$}\label{sect:homldecomp}

 Recall the homology decompositions of simply-connected spaces (cf. \cite{Arkbook}).
 Let $M$ be an $(n-1)$-connected $(2n+2)$ dimensional manifold with $H_\ast(M)$ given by (\ref{HM}). Denote by $M_{k}$ the $k$-th homology section in the homology decomposition of $M$. 
There are homotopy cofibrations in which the labelled attaching maps induce trivial homomorphism in homology (or called \emph{homologically trivial}):
\begin{equation}\label{Cofs}
  \begin{aligned}
   &\big(\bigvee_{i=1}^d S^n\big)\vee P^{n+1}(T)\xra{f}M_{n}\simeq\big(\bigvee_{i=1}^l S^n\big)\vee P^{n+1}(T)\to  M_{n+1},\\
&\bigvee_{i=1}^l S^{n+1}\xra{g}M_{n+1}\to  M_{n+2}=M_{2n+1},~~S^{2n+1}\xra{h}M_{2n+1}\to M.
\end{aligned}
\end{equation}

\begin{lemma}\label{lem:n+1}
  There is a homotopy equivalence
  \[M_{n+1}\simeq \big(\bigvee_{i=1}^l S^{n}\big)\vee \big(\bigvee_{i=1}^d S^{n+1}\big)\vee P^{n+1}(T)\vee P^{n+2}(T).\]
\begin{proof}
Consider the first homotopy cofibration in (\ref{Cofs}). It suffices to show that $f$, or equivalently each of the following components of the map $f$ is null-homotopic:
\begin{align*}
  f_j\colon &S^n_j\hookrightarrow \bigvee_{i=1}^d S^{n}\hookrightarrow \big(\bigvee_{i=1}^d S^{n}\big)\vee P^{n+1}(T)\xra{f}M_{n},~j=1,2\cdots,l; \\
  f_P\colon & P^{n+1}(T)\hookrightarrow  \big(\bigvee_{i=1}^d S^{n}\big)\vee P^{n+1}(T)\xra{f}M_{n}.
\end{align*}

Consider the commutative diagram of homotopy cofibrations:
\[\begin{tikzcd}[sep=scriptsize]
  \bigvee_{j=1}^tS^{n}\ar[d,"m_T"]\ar[r]&\ast\ar[r]\ar[d]&\bigvee_{i=1}^tS^{n+1}\ar[d,"\matwo{\ast}{m_T}"]\\
  \bigvee_{j=1}^tS^{n}\ar[r,"\ast"]\ar[d,"\iota"]& M_{n}\ar[r]\ar[d,equal]&M_{n}\vee \big(\bigvee_{j=1}^tS^{n+1}\big)\ar[d]\\
  P^{n+1}(T)\ar[r,"f_P"]& M_{n}\ar[r]&C_{f_P}
\end{tikzcd},\]   
where $m_T=\bigvee_jp_j^{r_j}$. Since $f$ is homologically trivial, so are $f_j$ and $f_P$. The Hurewizc theorem then implies that both $f_j$ and $f_P\circ \iota$ are null-homotopic, $j=1,2,\cdots,l$.
Thus the homotopy cofibration on third column implies that there is a homotopy equivalence 
\[C_{f_P}\simeq M_{n}\vee P^{n+2}(T),\]
which in turn implies that $f_P$ is null-homotopic.
\end{proof}
\end{lemma}

\begin{lemma}\label{lem:2n+1}
 There is a homotopy equivalence 
\[M_{2n+1}\simeq \big(\bigvee_{i=1}^dS^{n+1}\big)\vee P^{n+2}(T)\vee C_{\varphi_n},\]
where $\varphi_n\colon \bigvee_{i=1}^l S^{n+1}\to \big(\bigvee_{i=1}^lS^n\big)\vee P^{n+1}(T)$ is a homologically trivial map.
Furthermore, there holds homotopy equivalences 
\[ \Sigma C_{\varphi_2}\simeq C_{\varphi_3}, ~~  C_{\varphi_n}\simeq P^{n+1}(T_{\neq 2})\vee C_{\bar{\varphi}_n} \text{ for $n\geq 3$},\]
where $\bar{\varphi}_n\colon \bigvee_{i=1}^l S^{n+1}\to \big(\bigvee_{i=1}^lS^n\big)\vee P^{n+1}(T_2)$.

\begin{proof}
There is a homotopy cofibration $\bigvee_{i=1}^lS^{n+1}\xra{g}M_{n+1}\to M_{2n+1}$ with $M_{n+1}$ given by Lemma \ref{lem:n+1}. Consider the compositions 
  \begin{align*}
  g_{j}\colon & S^{n+1}\hookrightarrow \bigvee_{i=1}^lS^{n+1}\xra{g}M_{n+1}\to \big(\bigvee_{i=1}^dS^{n+1}\big)\to S^{n+1}_j;\\
  g_P\colon & S^{n+1}\hookrightarrow \bigvee_{i=1}^lS^{n+1}\xra{g}M_{n+1}\to P^{n+2}(T).
  \end{align*}
  Since $g$ is homologically trivial, so are $g_{j}$ and $g_P$.
  The Hurewicz theorem then implies that $g_{S,j}$ and $g_P$ are null-homotopic. 
  The homotopy equivalence for $M_{2n+1}$ in the Lemma then follows by Lemma \ref{lem:HL}.

For $n\geq 3$ and odd primes $p$, $\pi_{n+1}(P^{n+1}(p^r))=0$ implies the map $\varphi_n$ factors as the composition  
 \[\bigvee_{i=1}^lS^{n+1}\xra{\bar{\varphi}_n} \big(\bigvee_{i=1}^lS^n\big)\vee P^{n+1}(T_2)\hookrightarrow \big(\bigvee_{i=1}^lS^n\big)\vee P^{n+1}(T),\]
thus we have $C_{\varphi_n}\simeq  P^{n+1}(T_{\neq 2})\vee C_{\bar{\varphi}_n}.$
If $n=2$, we do once suspension to make $P^4(T_{\neq 2})$ split off $\Sigma C_{\varphi_2}$ and cancel the possible Whitehead products in $\varphi_2$; thus 
$\Sigma C_{\varphi_2}\simeq  P^4(T_{\neq 2})\vee C_{\bar{\varphi}_3}\simeq C_{\varphi_3}.$
\end{proof}
\end{lemma}

\begin{lemma}\label{lem:2n+1:n>2}
Let $T_2\cong\bigoplus_{i=1}^{t_2}\z{r_i}$ and let $n\geq 3$. If the Steenrod square $\Sq^2$ acts trivally on $H^n(C_{\varphi_n};\z{})$, then there is a homotopy equivalence 
\[C_{\varphi_n}\simeq \big(\bigvee_{i=1}^lS^n\big)\vee P^{n+1}(T)\vee \big(\bigvee_{i=1}^lS^{n+2}\big);\]
otherwise there exist integers $0\leq c_1\leq l,0\leq c_2\leq t_2$ with $0<c_1+c_2\leq l$ ($c_1=0$ or $c_2=0$ if $l=1$) such that 
\begin{multline*}
  C_{\varphi_n}\simeq \big(\bigvee_{i=1}^{l-c_1} S^{n}\big)\vee\big(\bigvee_{i=1}^{c_1} C^{n+2}_\eta\big)\vee \big(\bigvee_{i=1}^{l-c_1-c_2} S^{n+2}\big)\vee P^{n+1}(T_{\neq 2})\\
  \vee \big(\bigvee_{j=c_2+1}^{t_2} P^{n+1}(2^{r_j})\big) \vee \big(\bigvee_{j=1}^{t_2} C^{n+2}_{r_j}\big) .
\end{multline*} 

\begin{proof}
Write $\bar{\varphi}=\bar{\varphi}_n$ for simplicity. Note that $n\geq 3$ is the stable range for maps in $\pi_{n+1}\big((\bigvee_{i=1}^lS^n)\vee P^{n+1}(T_2)\big)$. Hence $\bar{\varphi}$ contains no Whitehead products
and can be represented by a matrix $\smatwo{A}{B}$, where $A$ is an $l\times l$ matrix representing the component 
 \[\bar{\varphi}_A\colon \bigvee_{i=1}^lS^{n+1}\xra{\bar{\varphi}} \big(\bigvee_{i=1}^lS^n\big)\vee P^{n+1}(T_2)\to \bigvee_{i=1}^lS^n,\]
 and $B$ is an $l\times t_2$ matrix representing the component 
 \[\bar{\varphi}_B\colon \bigvee_{i=1}^lS^{n+1}\xra{\bar{\varphi}} \big(\bigvee_{i=1}^lS^n\big)\vee P^{n+1}(T_2)\to P^{n+1}(T_2)\simeq \bigvee_{i=1}^{t_2}P^{n+1}(2^{r_i}).\]
 By Lemma \ref{lem:Moore1}, entries of the matrix $A$ are possibly $0,\eta$ and entries of $B$ are possible $0$ or $i_n\eta$.

 If $\Sq^2(H^n(C_\varphi);\z{})=0$, then $\Sq^2(H^n(C_{\bar{\varphi}});\z{})=0$. Lemma \ref{lem:Sq2} (\ref{Sq2-Changcpx}) then implies that $\smatwo{A}{B}$ is the zero matrix; that is, $\bar{\varphi}$ is null-homotopic, which proves the first homotopy equivalence. 

 If $\Sq^2(H^n(C_\varphi);\z{})\neq 0$, then $\Sq^2(H^n(C_{\bar{\varphi}});\z{})\neq 0$. By Lemma \ref{lem:Sq2} (\ref{Sq2-Changcpx}) we see that at least one entry of $\smatwo{A}{B}$ is nonzero.
 Observe that there are equivalences
 \begin{align*}
  (i)~&\matwo{\eta}{\eta}\sim \matwo{\eta}{0}\sim \matwo{0}{\eta}\colon S^{n+1}\to S^n\vee S^n,\\
  (ii)~&\matwo{\eta}{i_n\eta}\sim \matwo{\eta}{0}\colon S^{n+1}\to S^n\vee P^{n+1}(2^r),\\
  (iii)~&\matwo{i_n\eta}{i_n\eta}\sim \matwo{i_n\eta}{0}\colon S^{n+1}\to P^{n+1}(2^r)\vee P^{n+1}(2^s) \text{ for $r\geq s$}.
 \end{align*}
 By $(i)$ we firstly apply elementary transformations to the matrix $A$ to get a diagonal matrix $D$ with diagonal entries
 \[d_1=\cdots=d_{c_1}=\eta,~~d_{c_1+1}=\cdots=d_l=0.\] 
By $(ii)$ we then apply appropriate row transformations to $\smatwo{D}{B}$ to make the entries of $B$ below the nonzero $d_1=\cdots=d_{c_1}=\eta$ all zero. 
Denote the resulted matrix by $\smatwo{D}{C}$. The nonzero entries of $C$ only possibly appear in the last $l-c_1$ columns. Finally, by $(iii)$ we apply appropriate row transformations to $C$ to get a matrix $\smatwo{D}{E}$ such each column of $E$, and hence of $\smatwo{D}{E}$ contains at most one nonzero entry. 
 Let $c_2$ be the number of $i_n\eta$ in $E$, then the second homotopy equivalence in the Lemma follows.
\end{proof}
\end{lemma}

Note that we cannot determine the indices $r_1,\cdots,r_{c_2}$ among $r_1,\cdots,r_{t_2}$. 
Summarizing from Lemma \ref{lem:2n+1} and \ref{lem:2n+1:n>2} we get the following.
\begin{proposition}\label{prop:2n+1}
  Let $M$ be an $(n-1)$-connected $(2n+2)$-manifold with $H_\ast(M)$ given by (\ref{HM}), $n\geq 2$. Identify $\Sigma M_5$ with $M_7$. 
  There is a homotopy equivalence 
 \begin{multline*}
   M_{2n+1}\simeq \big(\bigvee_{i=1}^dS^{n+1}\big)\vee P^{n+2}(T)\vee \big(\bigvee_{i=1}^{l-c_1} S^{n}\big)\vee \big(\bigvee_{i=1}^{l-c_1-c_2} S^{n+2}\big)\vee P^{n+1}(T_{\neq 2})\\\vee\big(\bigvee_{i=1}^{c_1} C^{n+2}_\eta\big)\vee \big(\bigvee_{j=c_2+1}^{t_2}P^{n+1}(2^{r_j})\big)\vee\big(\bigvee_{j=1}^{c_2} C^{n+2}_{r_j}\big), 
 \end{multline*}
 where $0\leq c_1\leq l,0\leq c_2\leq t_2$ and $c_1+c_2\leq l$. Moreover, 
 \begin{enumerate}
  \item $c_1=c_2=0$ if and only if $\Sq^2$ acts trivially on $H^n(M;\z{})$.
  \item Suppose $l=1$ and $\Sq^2$ acts non-trivially on $H^n(M;\z{})$ (Replace $H^n(M;\z{})$ by $H^3(\Sigma M;\z{})$ for $n=2$). If for any $x\in H^n(M;\z{})$ with $\Sq^2(x)\neq 0$ and any $y\in \ker(\Sq^2)$, we have \(\beta_r(x+y)= 0\text{ for any $r\geq 1$},\) 
  then $(c_1,c_2)=(1,0)$; otherwise $(c_1,c_2)=(0,1)$.    
 \end{enumerate}
 
\end{proposition}

The following is immediate.
\begin{corollary}\label{cor:2n+1-local}
  Under the same assumptions in Proposition \ref{prop:2n+1}, there is a homotopy equivalence 
 \begin{multline*}
   M_{2n+1}\simeq_{\loca} \big(\bigvee_{i=1}^dS^{n+1}\big)\vee \big(\bigvee_{i=1}^{l} (S^{n}\vee S^{n+2})\big)\vee P^{n+1}(T)\vee P^{n+2}(T).
 \end{multline*}
\end{corollary}

As a byproduct, we note that Lemma \ref{lem:2n+1} and Lemma \ref{lem:2n+1:n>2} (or Proposition \ref{prop:2n+1}) help to understand the homotopy type of the loop space $\Omega M$ more clearly characterized by Chenery \cite{Chen23} or Theriault \cite{Th20}. The proof of the following Corollary is similar to that of \cite[Theorem 5.1]{Chen23}. 
\begin{corollary}\label{cor:Chen}
  Let $M$ be an $(n-1)$-connected $(2n+2)$-dimensional \PD ~complex whose integral homology is given by (\ref{HM}) with  $d>1$, $n=2$, or $n\geq 4, n\neq 7$. 
  Then there is a homotopy equivalence 
  \begin{align*}
    \Omega M&\simeq \Omega(S^{n+1}\times S^{n+1})\times \Omega  \big(\Omega(S^{n+1}\times S^{n+1})\ltimes X\big), 
  \end{align*}
where $X=\big(\bigvee_{i=1}^{d-2}S^{n+1}\big)\vee P^{n+2}(T)\vee C_{\varphi_n}$ is given by Lemma \ref{lem:2n+1} for $n=2$ and by Lemma \ref{lem:2n+1:n>2} for $n\geq 4, n\neq 7$.
\end{corollary}

\section{Proof of Theorem \ref{thm:6mfds}}\label{sect:n=2}
Let $M$ be a closed simply-connected $6$-manifold. By Lemma \ref{lem:2n+1}, there is a homotopy cofibration 
\[S^5\xra{h}M_5\simeq \big(\bigvee_{i=1}^dS^3\big)\vee P^4(T)\vee C_{\varphi_2}\to M,\]
where $\varphi_2\colon \bigvee_{i=1}^lS^3\to \big(\bigvee_{i=1}^lS^2\big)\vee P^3(T)$.
We may set 
\[h=\sum_{i=1}^dh_i^S+\sum_{j=1}^t+h_{\varphi_2}+\theta,\] 
where $h_i^S,  h_j^P$ and $h_{\varphi_2}$ are the compositions of $h$ with the canonical projections onto $S^3_i=S^3$, $P^4(p_i^{r_i})$ and $C_{\varphi_2}$, respectively; $\theta$ is a sum of all possible Whitehead products, $\Sigma \theta=0$.
Thus after suspension we get a homotopy cofibration 
\begin{equation*}
  S^6\xra{\bar{h}} \big(\bigvee_{i=1}^dS^4\big)\vee \big(\bigvee_{i=1}^tP^5(p_i^{r_i})\big)\vee \Sigma C_{\varphi_2}\to \Sigma M,
\end{equation*}
where $\bar{h}=\sum_{i=1}^d\Sigma h_i^S+\sum_{j=1}^t \Sigma h_j^P+\Sigma h_{\varphi_2}$. 
For our topological $6$-manifold $M$, the generalized splitting theorem of Wall \cite{Wall66} or Jupp \cite{Jupp73} implies a homotopy equivalence  (cf. \cite[Corollary 2.2]{Huang-arxiv}) 
\begin{equation}\label{eq:Wall}
  \Sigma M\simeq \big(\bigvee_{i=1}^dS^4\big)\vee \Sigma M'.
\end{equation} 
Denote $T_2[c_2]=\bigoplus_{i=c_2+1}^{t_2}\z{r_i},~~T[c_2]=T_2[c_2] \oplus T_{\neq 2}.$
By Proposition \ref{prop:2n+1} with $n=2$, $\Sigma M'$ is the homotopy cofibre of a map 
\begin{equation}\label{def:hbar}
  \begin{multlined}
  S^6\xra{\bar{h}'} \Sigma M_5'\xra[\simeq]{e}P^5(T)\vee \big(\bigvee_{i=1}^{l-c_1} S^{3}\big) \vee \big(\bigvee_{i=1}^{l-c_1-c_2} S^{5}\big)
  \vee\big(\bigvee_{i=1}^{c_1} C^{5}_\eta\big)\\\vee P^4(T[c_2])\vee \big(\bigvee_{j=1}^{c_2} C^{5}_{r_j}\big)=:\Sigma W'.
\end{multlined}
\end{equation}

\begin{lemma}\label{lem:n=2:susp}
 Let $Z$ be an element or a wedge sum of any two elements of $\{S^2,S^4,C^4_\eta,C^4_{r_j},P^3(2^{r_i})\}$ with $\z{r_i}$ direct summands of $T_2[c_2]$. Let $p'_{Z}\colon W'\to Z$ be the canonical projection onto $Z$. Then all the compositions
 \[S^6\xra{\Sigma \bar{h}'} \Sigma M_5'\xra{e} \Sigma W' \xra{\Sigma p'_Z} \Sigma Z\]
are suspensions of certain maps.
 \begin{proof}
Write $\Sigma W'=P^5(T)\vee P^4(T_{\neq 2})\vee \Sigma X$ with 
\[X=\big(\bigvee_{i=1}^{l-c_1} S^{2}\big) \vee \big(\bigvee_{i=1}^{l-c_1-c_2} S^{4}\big)\vee\big(\bigvee_{i=1}^{c_1} C^{4}_\eta\big)\vee \big(\bigvee_{i=c_2+1}^{t_2}P^3(2^{r_i})\big)\vee \big(\bigvee_{j=1}^{c_2} C^{4}_{r_j}\big).\]
Clearly the projection $p=(\Sigma p'_Z)\circ e\colon \Sigma M_5'\to\Sigma Z$ factors through $\Sigma X$. Let $H$ be the second James-Hopf invariant, then $H(\Sigma p'_Z)=0$. Consider the following commutative diagram 
\[\begin{tikzcd}
  {[M_5',Z]}\ar[r,"E"]&{[\Sigma M_5',\Sigma Z]}\ar[r,"H"]&{[\Sigma M_5',\Sigma Z\wedge Z]}\\
  &{[\Sigma W',\Sigma Z]}\ar[u,"e^\sharp","\cong"swap]\ar[r,"H"]&{[\Sigma W',\Sigma Z\wedge Z]}\ar[u,"e^\sharp","\cong"swap]
\end{tikzcd},\]
 where $E(\alpha)=\Sigma \alpha$ is the suspension homomorphism. 
We have 
\[H(p)=H\circ e^\sharp(\Sigma p')=e^\sharp H(\Sigma p')=0.\]
Since $M_5'$ has dimension $4$, the top row is the generalized \emph{EHP} exact sequence (cf. \cite[Section 4.K]{hatcherbook}). 
Then it follows by exactness that $p$ is the suspension of some map $M_5'\to Z$. 
Thus the composition $p\circ \Sigma \bar{h}'$ is a suspension.
 \end{proof}
\end{lemma}

In \cite{CS22} Cutler and So showed that $\hbar$ contains trivial Whitehead product component $W$ if $H_\ast(M)$ contains no $2$-torsions; i.e., $T_2=0$. When $T_2\neq 0$, after once-suspension the possible Whitehead products \[ \theta_{\varphi}\colon S^6\to \Sigma C_{\varphi}\]
lie in the groups $\pi_6(\Sigma Y_1\vee \Sigma Y_2)$, where 
\(Y_1,Y_2\in \{S^2,S^4,P^3(2^{r_i}),C^4_\eta,C^4_{r_j}\}.\)  
 
\begin{proposition}\label{prop:n=2:key}
  The attaching map $\hbar=e\circ \Sigma \bar{h}'\colon S^6\to \Sigma W'$ in (\ref{def:hbar}) desuspends and contains no Whitehead product components.
  \begin{proof}
 By the work of \cite{CS22}, it suffices to consider the case $T=T_2$.  
    Note that the second James-Hopf invariant $H$ satisfies $H\circ E=0$, where $E(\alpha)=\Sigma \alpha$.
It follows that if there exist elements in $\pi_6(\Sigma Y_1\vee \Sigma Y_2)$ satisfying $H(\alpha)\neq 0$, then $\alpha\neq \Sigma \alpha'$, and hence $\alpha$ is not a component of $\theta_\varphi$. Then the Proposition follows by Lemma \ref{lem:htpgrps-smash}, \ref{lem:HW} and \ref{lem:n=2:susp}.
  \end{proof}
\end{proposition}

Due to \cite{CS22}, $P^5(T_{\neq 2})$ and $P^4(T_{\neq 2})$ split off $\Sigma M'$.
By the group structures and generators in Lemma \ref{lem:Moore1}, \ref{lem:Moore-n+3} and \ref{lem:htpgrps-Chang-dim6}, we observe that $\pi_6(C^5_\eta)$ contains no suspended elements and the suspended generators of $\pi_6(P^5(2^s))$, $\pi_6(S^3)$, $\pi_6(P^4(2^r))$,  $\pi_6(C^5_r)$ are respectively
\[\tilde{\eta}_r,~i_4\eta^2;\quad 6\nu'=\eta^3;\quad \tilde{\eta}_r\eta;\quad i_P\tilde{\eta}_r\eta.\] 
By Proposition \ref{prop:n=2:key} and Lemma \ref{lem:HL}, we then get that 
\[\Sigma M'\simeq P^5(T_{\neq 2})\vee P^4(T_{\neq 2})\vee \big(\bigvee_{i=1}^{c_1}C^5_\eta\big)\vee C_{\hbar},\] 
where $C_{\hbar}$ is the homotopy cofibre of a map 
\[\hbar\colon S^6\to  P^5(T_2)\vee\big(\bigvee_{i=1}^{l-c_1} S^{3}\big)\vee \big(\bigvee_{i=1}^{l-c_1-c_2} S^{5}\big)\vee P^4(T_2[c_2])\vee \big(\bigvee_{j=1}^{c_2}  C^{5}_{r_j}\big).\]
Since $2\tilde{\eta}_1=i_4\eta^2$ (\ref{eq:eta}), we may put 
\begin{equation}\label{eq:hbar}
  \small  \hbar= \sum_{i=1}^{t_2} (x_i\cdot \tilde{\eta}_{r_i}\!+\!\varepsilon_i\cdot i_4\eta^2)+\sum_{i=1}^{c_1}y_i\cdot \eta^3+\sum_{i=1}^{l-c_1-c_2}z_i\cdot \eta
    +\sum_{j=c_2+1}^{t_2}s_j\cdot \tilde{\eta}_{r_{j}}\eta+\sum_{j=1}^{c_2}  t_j \cdot i_P\tilde{\eta}_{r_j}\eta,
\end{equation}
where all the coefficients belong to $\{0,1\}$, and $c_1=c_2=0$ if and only if the Steenrod square $\Sq^2$ acts trivially on $H^2(M;\z{})$.
We discuss the homotopy types of $C_{\hbar}$ and $\Sigma M'$ by the following two propositions.  In their proofs we shall frequently use the phrase ``we may assume that'' to mean that the original situations (or maps) is equivalent to the latter ones by doing suitable matrix operations, and their mapping cones are homotopy equivalent.

\begin{proposition}\label{prop:SM'-Sq2eq0}
  Suppose that $\Sq^2$ acts trivially on $H^4(M;\z{})$.  
   \begin{enumerate}[1.]
    \item\label{Sq2eq0-1} If the secondary operation $\Theta$ acts trivially on $H^3(M;\z{})$, then  
    \begin{multline*}
      \Sigma M'\simeq  P^5(T)\vee \big(\bigvee_{i=2}^{l-c_1} S^{3}\big) 
      \vee \big(\bigvee_{i=1}^{l-c_1-c_2} S^{5}\big)\vee\big(\bigvee_{i=1}^{c_1} C^{5}_\eta\big)\vee \big(\bigvee_{j=1}^{c_2}  C^{5}_{r_j}\big)\\ \vee P^4(T[c_2])\vee (S^3\cup_{\epsilon\cdot\eta^3}e^7),
     \end{multline*}
    where  $\epsilon\in\{0,1\}$, and $\epsilon=0$ if $\mathbb{T}\big(H^2(M;\z{})\big)=0$ or  $l=c_1=1$.
    \item \label{Sq2eq0-2} If $\Theta$ acts non-trivially on $H^3(M;\z{})$, then the homotopy type of $\Sigma M'$ can be distinguished as follows.
     \begin{enumerate}
      \item\label{Sq2eq0-2a} If for any $u,v\in H^4(\Sigma M;\z{})$ with $\Theta(u)\neq 0$, $\Theta(v)=0$, there holds $u+v\notin \im(\beta_s)\text{ for any } s\geq 1;$
      while there exist $u',v'\in H^4(\Sigma M;\z{})$ satisfying $\Theta(u')\neq 0,\Theta(v')=0$ and $ \beta_r(u'+v')\neq 0,$ then there is a homotopy equivalence 
      \begin{multline*}
        \Sigma M'\simeq  \big(\bigvee_{i=1}^{l-c_1} S^{3}\big) \vee \big(\bigvee_{i=1}^{l-c_1-c_2} S^{5}\big)
        \vee\big(\bigvee_{i=1}^{c_1} C^{5}_\eta\big)\vee P^4(T[c_2])\vee \big(\bigvee_{j=1}^{t_2}  C^{5}_{r_j}\big)\\ \vee P^5\big(\frac{T}{\z{r_{j_1}}}\big)\vee (P^5(2^{r_{j_1}})\cup_{i_4\eta^2}e^7),
       \end{multline*} 
where $r_{j_1}$ is the maximum of $r_{j}$ such that $\beta_{r_j}(u'+v')\neq 0$. 
      \item\label{Sq2eq0-2b} If there exist $u,v\in H^4(\Sigma M;\z{})$ satisfying $\Theta(u)\neq 0$, $\Theta(v)=0$, $u+v\in\im(\beta_r)$ for some $r$,
      then we have 
      \begin{multline*}
        \Sigma M'\simeq P^5(T)\vee \big(\bigvee_{i=1}^{l-c_1} S^{3}\big) 
        \vee \big(\bigvee_{i=1}^{l-c_1-c_2} S^{5}\big)\vee\big(\bigvee_{i=1}^{c_1} C^{5}_\eta\big)\vee \big(\bigvee_{j=1}^{t_2}  C^{5}_{r_j}\big)\\
        \vee P^4\big(\frac{T[c_2]}{\z{r_{j_2}}}\big) \vee (P^4(2^{r_{j_2}})\cup_{\tilde{\eta}_{r_{j_2}}\eta}e^7), \text{ or }\\[1ex]
        \Sigma M'\simeq P^5(T)\vee \big(\bigvee_{i=1}^{l-c_1} S^{3}\big) \vee \big(\bigvee_{i=1}^{l-c_1-c_2} S^{5}\big)
        \vee\big(\bigvee_{i=1}^{c_1} C^{5}_\eta\big)\vee \big(\bigvee_{j_2\neq j=1}^{t_2}  C^{5}_{r_j}\big)\\ \vee P^4(T[c_2])\vee (C^5_{r_{j_2}}\cup_{i_P\tilde{\eta}_{r_{j_2}}\eta}e^7),
       \end{multline*} 
where $r_{j_2}$ is the minimum of $r_j$ such that $u+v\in \im(\beta_{r_j})$. The second homotopy equivalence is impossible if $l=c_1=1$.
     \end{enumerate}
   \end{enumerate}
\begin{proof}
Consider the equation (\ref{eq:hbar}) for $\hbar$. The assumption that $\Sq^2$ acts trivially on $H^5(\Sigma M;\z{})$ indicates $x_i=z_i=0$ for all $i$.
Thus we have 
\begin{equation}\label{hbar:Sq2eq0}
  \hbar= \sum_{i=1}^{t_2} \varepsilon_i\cdot i_4\eta^2+\sum_{i=1}^{c_1}y_i\cdot \eta^3+\sum_{j=c_2+1}^{t_2}s_j\cdot \tilde{\eta}_{r_j}\eta+\sum_{j=1}^{c_2}  t_j \cdot i_P\tilde{\eta}_{r_j}\eta,
\end{equation}
where all coefficients lie in $\Z/2$.

(1) By Lemma \ref{lem:Theta}, the condition that $\Theta\big(H^4(\Sigma M;\z{})\big)=0$ implies 
\[\varepsilon_i=s_j=t_j=0,~\forall~i,j.\]
If the tertiary operation $\mathbb{T}$ acts trivially on $H^3(\Sigma M;\z{})$, then $y_i=0$ for all $i\leq c_1$;
otherwise we may assume that $y_1=1$ and $y_i=0$ for $i=2,\cdots,c_1$. 
Thus 
\[\hbar=\epsilon\cdot \eta^3\colon S^6\to S^3.\]  
The homotopy equivalence in (\ref{Sq2eq0-1}) follows.

(2) By Lemma \ref{lem:Theta}, the condition that $\Theta\big(H^4(\Sigma M;\z{})\big)\neq 0$ implies at least one of $\varepsilon_i,s_j,t_j$ equals $1$. 

(a) The assumptions in (\ref{Sq2eq0-2a}) guarantee that in (\ref{hbar:Sq2eq0}),
\[\varepsilon_i=1 \text{ for some $i$ and } s_j=t_j=0\text{ for all $j$}.\]
Since $\smatwo{i_4\eta^2}{i_4\eta^2}\sim \smatwo{0}{i_4\eta^2}\colon S^6\to P^5(2^r)\vee P^5(2^s)$ for $r\leq s$,
we may assume that $\varepsilon_{j_1}=1$ and $\varepsilon_j=0$ for $j_1\neq j\leq t_2$, where $j_1$ is described in (\ref{Sq2eq0-2a}). Dual to the formula $q_4\tilde{\eta}_r= \eta$ (\ref{eq:eta}), there exists a map $\bar{\eta}_r\colon P^{5}(2^r)\to S^3$ satisfying $\bar{\eta}_r\circ i_{4}=\eta$, $2\bar{\eta}_1= \eta^2q_{5}.$
By the equivalence 
\[\matwo{i_4\eta^2}{\eta^3}\sim \matwo{i_4\eta^2}{0}\colon S^6\to P^5(2^r)\vee S^3\]
we can further assume that $y_i=0$ for all $i$. Thus we get 
\[\hbar=i_4\eta^2\colon S^6\to P^5(2^{r_{j_1}}),\]
and therefore (\ref{Sq2eq0-2a}) is proved.

(b) Under the conditions in (\ref{Sq2eq0-2b}), we have $s_j=1$ or $t_j=1$ for some $j$. By (\ref{eq:chi-eta}) and (\ref{eq:C_rs}),
there hold equivalences 
\begin{align*}
  \matwo{\tilde{\eta}_r\eta}{\beta_X}\sim \matwo{\tilde{\eta}_r\eta}{0}\colon S^6\to P^4(2^r)\vee X,\quad \matwo{i_P\tilde{\eta}_r\eta}{\beta_Y}\sim \matwo{i_P\tilde{\eta}_r\eta}{0}\colon S^6\to  C^5_r\vee Y,
\end{align*}
where $\beta_X=i_4\eta^2, \eta^3, \tilde{\eta}_s\eta$ for $X=P^5(2^s),S^3,P^4(2^s)(r\leq s)$; $\beta_Y=i_4\eta^2, \eta^3, \tilde{\eta}_s\eta$ for $Y=P^5(2^s),S^3, C^5_s(r\leq s)$, respectively.
Thus we may assume that $\varepsilon_i=y_i=0$ for all $i$; moreover, if $s_j=1$ for some $c_2+1\leq j\leq t_2$, then we can further assume that 
\[s_{j_2}=1\text{ and }s_j=0 \text{ for }j\neq j_2,\]
where $j_2$ is the index such that $r_{j_2}$ is the minimum of $r_j$ among the original indices $s_j=1$.
if $t_j=1$ for some $1\leq j\leq c_2$, then we can further assume that 
\[t_{j_3}=1\text{ and }t_j=0\text{ for } j\neq j_3,\]
where $j_3$ is the index such that $r_{j_3}$ is the minimum of $r_j$ among the original indices $t_j=1$.
By Example \ref{ex} we have  
\[\matwo{\tilde{\eta}_{r_{j_2}}\eta}{i_P\tilde{\eta}_{r_{j_3}}\eta}\sim \matwo{\tilde{\eta}_{r_{j_2}}\eta}{0}\text{ if }r_{j_2}\leq r_{j_3},\quad \matwo{\tilde{\eta}_{r_{j_2}}\eta}{i_P\tilde{\eta}_{r_{j_3}}\eta}\sim 
\matwo{0}{i_P\tilde{\eta}_{r_{j_3}}\eta}\text{ if }r_{j_2}>r_{j_3}.\]
Thus we get 
\[\hbar=\left\{\begin{array}{ll}
  \quad\tilde{\eta}_{r_{j_2}}\colon S^6\to P^4(2^{r_{j_2}}) &\text{ if } r_{j_2}\leq r_{j_3};\\[1ex]
  i_P\tilde{\eta}_{j_3}\eta\colon S^6\to C^{5}_{r_{j_3}} &\text{ otherwise.}
\end{array}\right.\]
The proof of (\ref{Sq2eq0-2b}) is completed.
\end{proof}
\end{proposition}

\begin{proposition}\label{prop:SM'-Sq2neq0}
 Suppose that $\Sq^2$ acts non-trivially on $H^4(M;\z{})$.
\begin{enumerate}[1.]
  \item\label{Sq2neq0-1} If for any $u\in H^5 (\Sigma M;\z{})$ with $\Sq^2(u)\neq 0$ and any $v\in \ker(Sq^2)$, there hold $u+v\notin \im (\beta_r)$ for any $r\geq 1$,
then there is a homotopy equivalence 
\begin{multline*}
 \small\Sigma M'\!\simeq \! P^5(T)\vee (\bigvee_{i=1}^{l-c_1} S^{3})\vee (\bigvee_{i=2}^{l-c_1-c_2} S^{5})
 \vee (\bigvee_{i=1}^{c_1} C^{5}_\eta)\vee P^4(T[c_2])\vee (\bigvee_{j=1}^{c_2}  C^{5}_{r_j})  \vee C^7_\eta .
\end{multline*}
This possibility doesn't exist if $l=1$ and $\Sq^2(H^2(M;\z{}))$.
  \item\label{Sq2neq0-2} If there exist $u,v\in H^5(\Sigma M;\z{})$ such that 
  \[\Sq^2(u)\neq 0,~~\Sq^2(v)= 0,~~u+v\in \im (\beta_{r_i}) \text{ for some }r_i,\]
  then there is a homotopy equivalence 
  \begin{multline*}
    \Sigma M'\simeq  P^5\big(\frac{T}{\z{r_{i_1}}}\big)\vee (\bigvee_{i=1}^{l-c_1} S^{3}) \vee (\bigvee_{i=1}^{l-c_1-c_2} S^{5})
    \vee (\bigvee_{i=1}^{c_1} C^{5}_\eta)\vee P^4(T[c_2])\\ \vee (\bigvee_{j=1}^{c_2}  C^{5}_{r_j}) \vee \big(P^5(2^{r_{i_1}})\cup_{\tilde{\eta}_{r_{i_1}}}e^7\big),
  \end{multline*}
  where $r_{i_1}$ is the minimum of $r_i$ such that $u+v\in \im (\beta_{r_i})$.
 
\end{enumerate}

\begin{proof}
By Lemma \ref{lem:Sq2} and \cite[Lemma 3.2 (2)]{lipc23}, the assumption that $\Sq^2$ acts non-trivally on $H^4(M;\z{})$, or equivalently on $H^5(\Sigma M;\z{})$ implies that at least one of $x_i$ and $z_i$ equals $1$ in the expression (\ref{eq:hbar}).

(1) If the conditions in (\ref{Sq2neq0-1}) hold, then $x_i=0$ for all $i\leq t_2$ and $z_i=1$ for at least one $i\leq l-c_1-c_2$. Clearly we may assume that there exists exactly one such $i$, say $z_1=1$ and $z_i=0$ for $2\leq i\leq l-c_1-c_2$.
There hold obvious equivalences 
\[ \matwo{\eta}{\beta_X}\sim \matwo{\eta}{0}\colon S^6\to S^5\vee X,\]
where $\beta_X=\eta^3,i_5\eta^2,\tilde{\eta}_r\eta,i_P\tilde{\eta}_r\eta$ for $X=S^3,P^5(2^r),P^4(2^r),C^5_r$, respectively.
Hence we may assume that $\varepsilon_i=y_i=s_i=t_i=0$ for all $i$. 
Thus we get  
\[\hbar=\eta\colon S^6\to S^5.\]
The homotopy equivalence in the Proposition (\ref{Sq2neq0-1}) follows. 

(2) If the conditions in (\ref{Sq2neq0-2}) hold, then we have $x_i=1$ for some $i\leq t_2$. By (\ref{eq:chi-eta}) there holds an equivalence 
\[ \matwo{\tilde{\eta}_r}{\tilde{\eta}_s}\sim \matwo{\tilde{\eta}_r}{0}\colon S^6\to P^5(2^r)\vee P^5(2^s) \text{ for }r\leq s.\]
Thus we may assume that $x_{i_1}=1$ and $x_i=0$ for $i_1\neq i\leq t_2$, where 
$r_{i_1}$ is the minimum of $r_i$ with $u+v\in \im(\beta_{r_i})$. The formula $q_5\tilde{\eta}_{r_1}=\eta$ implies 
\[ (1_P+i_4\eta q_5)(\tilde{\eta}_{r_1}+i_4\eta^2)=\tilde{\eta}_{r_1}.\]
Since $1_P+i_4\eta q_5$ is a self-homotopy equivalence of $P^5(2^{r_{i_1}})$, we may assume that $\varepsilon_{i_1}=0$.
The formula $q_5\tilde{\eta}_r=\eta$ also implies the following equivalences 
\[ \matwo{\tilde{\eta}_r}{\gamma_Y}\sim \matwo{\tilde{\eta}_r}{0}\colon S^6\to P^5(2^r)\vee Y,\]
where $\gamma_Y=i_4\eta^2,\eta,\tilde{\eta}_s\eta,i_P\tilde{\eta}_s\eta$ for $Y=P^5(2^s), S^5,P^4(2^s), C^5_s$, respectively.
Hence we may further assume that $\varepsilon_i=y_i=z_i=s_j=t_j=0$ for all $i,j$.
Thus we get 
\[\hbar=\tilde{\eta}_{r_{i_1}}\colon S^6\to P^5(2^{r_{i_1}}),\] 
which completes the homotopy equivalence in (\ref{Sq2neq0-2}). 
\end{proof}
\end{proposition}

\begin{proof}[Proof of Theorem \ref{thm:6mfds}]
Recall that for our $6$-manifolds $M$, the Wu formula implies that the second Stiefel-Whitney class $\omega_2(M)=0$ if and only if $\Sq^2$ acts trivially on $H^4(M;\z{})$.
The discussion on the homotopy types of $\Sigma M$ follows by (\ref{eq:Wall}), Proposition \ref{prop:SM'-Sq2eq0} and \ref{prop:SM'-Sq2neq0}.
\end{proof}

\section{Proof of Theorem \ref{thm:8-mflds}}\label{sect:n=3}
Let $M$ be a $2$-connected $8$-dimensional \PD complex with $H_\ast(M)$ given by (\ref{HM}). In this section we study the homotopy decompositions of $\Sigma M$ after localization away from $2$. For convenience, we assume that $T$ contains no $2$-torsion.

By Corollary \ref{cor:2n+1-local} and Lemma \ref{lem:HL}, after localization away from $2$ the attaching map $h$ of the top cell of $M$ factors as the composition
\[S^{7}\xra{h'}X \hookrightarrow \big(\bigvee_{i=1}^l(S^3\vee S^{5})\big)\vee X,~~X=\big(\bigvee_{i=1}^dS^{4}\big)\vee  P^{4}(T)\vee P^{5}(T).\]
Thus we have a homotopy equivalence 
\begin{equation}\label{eq:p-odd}
  \Sigma M\simeq_{\loca} \big(\bigvee_{i=1}^l(S^{4}\vee S^{6})\big)\vee C_{h''},
\end{equation}
where $h''\colon S^8\to\big(\bigvee_{i=1}^dS^{4}\big)\vee  P^{4}(T)\vee P^{5}(T)$ is a suspension map and contains no Whitehead products.
Set 
\(T=\bigoplus_{u=1}^{\bar{t}}\Z/3^{r_u}\oplus T_{\geq 5}\) for some $0\leq \bar{t}\leq k$. 
By Lemma \ref{lem:p-odd:n=3} we may put 
\begin{equation}\label{eq:h''}
  h''=_{\loca}\sum_{i=1}^da_i\cdot\alpha_1(5)+\sum_{j=1}^{\bar{t}}b_j\cdot \tilde{\alpha}_1(5)+\sum_{k=1}^{\bar{t}}c_k\cdot i_5\alpha_1(5),
\end{equation}
where $a_i,b_j,c_k\in\Z/3.$ Here we use the convention that $b_j=c_k=0$ if $\bar{t}=0$.

\begin{lemma}\label{lem:p-odd:n=3:P1=0}
 Let $h''$ be the map given by (\ref{eq:h''}). If $\PP$ acts trivially on $H^5(C_{h''};\Z/3)$, then there is a homotoppy equivalence
 \[C_{h''} \simeq_{\loca} \big(\bigvee_{i=1}^dS^5\big)\vee P^5(T)\vee P^6(T).\]
 \begin{proof}
   Since $\PP$ acts trivially on $H^5(C_{h''};\Z/3)$, Lemma \ref{lem:p-odd:n=3} and \ref{lem:p-odd:n=3:P1} imply that in the equation (\ref{eq:h''}),
   \[a_i=b_j=c_k=0,~\forall~i,j,k.\]
Then we get the homotopy equivalence by Lemma \ref{lem:HL}.
 \end{proof}
\end{lemma}

\begin{proposition}\label{prop:p-odd:n=3}
 Let $h''$ be the map given by (\ref{eq:h''}) and assume that $\PP$ acts non-trivially on $H^5(C_{h''};\Z/3)$. 
 \begin{enumerate}
   \item If for any $u\in H^5(C_{h''};\Z/3)$ with $\PP(u)\neq 0$ and any $v\in \ker(\PP)$, there hold $\beta_r(u+v)=0$ and $u+v\notin\im(\beta_s)$ for all $r,s\geq 1,$
   then there is a homotopy equivalence 
   \[C_{h''}\simeq_{\loca}(\bigvee_{i=2}^dS^5)\vee P^5(T)\vee P^6(T)\vee (S^5\cup_{\alpha_1(5)}e^9).\]

   \item If there exist $u\in H^5(C_{h''};\Z/3)$ with $\PP(u)\neq 0$ and $v\in \ker(\PP)$ such that $u+v\in \im(\beta_{r})~\text{ for some $r$},$
   then 
   \[C_{h''}\simeq_{\loca}(\bigvee_{i=2}^dS^5)\vee P^5\big(\frac{T}{\Z/3^{r_{j_1}}}\big)\vee P^6(T)\vee  P^{5}(3^{r_{j_1}})\cup_{\tilde{\alpha}_1(5)}e^9,\]
   where $r_{j_1}$ is the minimum of $r_j$ such that $u+v\in \im(\beta_{r_j})$.

  \item If for any $u\in H^5(C_{h''};\Z/3)$ with $\PP(u)\neq 0$ and any $v\in \ker(\PP)$, there hold 
  \(u+v\notin \im(\beta_s)\) for all $s\geq 1$;
  while there exist $u'\in H^5(C_{h''};\Z/3)$ with $\PP(u')\neq 0$ and $v'\in \ker(\PP)$ such that $\beta_r(u'+v')\neq 0~\text{ for some $r\geq 1$},$
then 
\[ C_{h''}\simeq_{\loca}(\bigvee_{i=2}^dS^5)\vee P^5(T)\vee P^6\big(\frac{T}{\Z/3^{r_{j_2}}}\big)\vee P^{6}(3^{r_{j_2}})\cup_{i_5\alpha_1(5)}e^9,\]
where $r_{j_2}$ is the maximum of $r_j$ such that $\beta_{r_j}(u'+v')\neq 0$.
 \end{enumerate}

\begin{proof}
Since $\PP$ acts non-trivially on $H^5(C_{h''};\Z/3)$, by Lemma \ref{lem:p-odd:n=3:P1} we see that in (\ref{eq:h''}) at least one of $a_i,b_j,c_j$ is non-zero. 
The relation $q_5\tilde{\alpha}_1(5)\simeq \alpha_1(5)$ implies the following equivalences
\begin{equation}\label{eq:n=3}
  \begin{aligned}
 \matwo{\alpha_1(5)}{\tilde{\alpha}_1(5)}&\sim \matwo{0}{\tilde{\alpha}_1(5)}\colon S^8\to S^5\vee P^5(3^r),\\
 \matwo{\tilde{\alpha}_1(5)}{i_5\alpha_1(5)}&\sim \matwo{\tilde{\alpha}_1(5)}{0}\colon S^8\to P^5(3^r)\vee P^6(3^s),\\
 \matwo{\alpha_1(5)}{i_5\alpha_1(5)}&\sim \matwo{\alpha_1(5)}{0}\colon S^8\to S^5\vee P^6(3^r).
\end{aligned}
\end{equation}

(1) Under the additional condition,  Lemma \ref{lem:p-odd:n=3:P1} implies  
\[b_j=c_k=0,~\forall~j,k\leq \bar{t}.\] 
Then by Lemma \ref{lem:HL} we have a homotopy equivalence 
 \[C_{h''}\simeq_{\loca} C_{h'''}\vee P^5(T)\vee P^6(T)\]
for some map $h'''=\sum_{i=1}^d a_i\cdot\alpha_1(5)\colon S^8\to \bigvee_{i=1}^dS^5$, $a_i\in\Z/3$. Since $\matwo{\nu_5}{\nu_5}\sim \matwo{\nu_5}{0}\sim \matwo{0}{\nu_5}$, we may assume that $a_1=1,a_2=\cdots=a_d=0$; thus we have 
\[C_{h'''}\simeq (\bigvee_{i=2}^dS^5)\vee (S^5\cup_{\alpha_1(5)}e^9).\]
Therefore there is a homotopy equivalence 
\[C_{h''}\simeq_{\loca}(\bigvee_{i=2}^dS^5)\vee P^5(T)\vee P^6(T)\vee  (S^5\cup_{\alpha_1(5)}e^9).\]

(2) Under the additional condition we have 
\[\bar{t}\geq 1,~b_j=\pm 1\text{ for some $j$},\]  
The equivalence relations in (\ref{eq:n=3}) then imply that 
\[a_i=c_k=0,~\forall~i=1,\cdots,d,~t=1,\cdots,\bar{t}.\] 
For $s\geq r$, the relation $B(\chi^r_s)\tilde{\alpha}_1(5)=\tilde{\alpha}_1(5)$ in Lemma \ref{lem:p-odd:n=3} implies an equivalence 
\[\matwo{\tilde{\alpha}_1(5)}{\tilde{\alpha}_1(5)}\sim \matwo{\tilde{\alpha}_1(5)}{0}\colon S^8\to P^5(3^r)\vee P^5(3^s),~~s\geq r.\]
Thus we may further assume that 
\[b_{j_1}=\pm 1,\quad b_j=0,~\forall~b_{j_1}<b_j\leq b_{\bar{t}}.\] 
Then the homotopy equivalence in (2) follows by Lemma \ref{lem:HL}.

(3) The additional condition implies 
\[ \bar{t}\geq 1,~a_i=b_j=0,~\forall~i\leq d,j\leq \bar{t},\]  
and hence $c_k=\pm 1$ for some $t\leq \bar{t}$.  For $r\geq s$, the relation 
$B(\chi^r_s)i_5\simeq i_5$ in (\ref{eq:chi})
implies the equivalence
\[\matwo{i_5\alpha_1(5)}{i_5\alpha_1(5)}\sim \matwo{i_5\alpha_1(5)}{0}\colon S^5\to P^6(3^r)\vee P^6(3^s), ~r\geq s.\]
Thus we may assume that 
\[c_{j_2}=\pm 1,\quad c_k=0, ~\forall~1\leq c_k<c_{j_2}.\]
The homotopy equivalence in (3) then follows by Lemma \ref{lem:HL}.
\end{proof}
\end{proposition}

\begin{proof}[Proof of Theorem \ref{thm:8-mflds}]
  Combine (\ref{eq:p-odd}), Lemma \ref{lem:p-odd:n=3:P1=0} and Proposition \ref{prop:p-odd:n=3}.
\end{proof}

\section{The weak second James-Hopf invariant} \label{sect:JH}
Let $X,Y$ be connected CW-complexes.
Recall that the Hilton-Milnor theorem states that there is a homotopy equivalence 
\[\Omega \Sigma (X\vee Y)\simeq \Omega\Sigma X\times \Omega\Sigma Y\times \Omega\Sigma (\vee_{i\geq 1}X^{\wedge i}\wedge Y).\]
Let $\iota_1\colon X\to X\vee Y$ and $\iota_2\colon Y\to X\vee Y$ be the canonical inclusion into wedge sum, respectively.
Let $E=E_X\colon X\to \Omega \Sigma X$ be the suspension map, which is the inclusion map if we identify $\Omega\Sigma X$ with the James construction $J(X)$ \cite{James55}.
Recall the following duality isomorphism 
\[\Omega_0\colon [\Sigma X,Y]\to [X,\Omega Y],\quad\Omega_0f=(\Omega f)\circ E_X.\]
For maps $f\colon \Sigma X\to Z,g\colon \Sigma Y\to Z$, the Whitehead product 
\[[f,g]\colon \Sigma X\wedge Y\to Z\] is the duality of the Samelson product 
\[[\Omega_0 f,\Omega_0g]^S\colon X\wedge Y\to \Omega Z; \]
that is, $\Omega_0[f,g]=[\Omega_0f,\Omega_0g]^S$.
Let $X$ be a connected CW co-$H$-space.  Recall that the second James-Hopf invariant $H=H_2$ is the composition 
\[H\colon \Omega \Sigma X\xra{\Omega \Sigma (\iota_1+\iota_2)} \Omega \Sigma (X\vee X)\xra{p_{[\iota_1,\iota_2]}}\Omega\Sigma (X\wedge X),\]
where $p_{[\iota_1,\iota_2]}$ is the canonical projection after pre-composing with the homotopy equivalence given by the Hilton-Milnor theorem.

For a map $f\colon X\to Y$, denote by $f_E$ the composition 
\[f_E=E_Y\circ f=(\Omega\Sigma f)\circ E_X=\Omega_0(\Sigma f)\colon X\to \Omega \Sigma Y.\] 
Let $\iota_i\colon S^{i+1}\to S^2\vee S^3$, $i=1,2$.
There hold isomorphisms 
\begin{align*}
  \pi_5(\Omega\Sigma (S^2\vee S^3)^{\wedge 2})&\cong \pi_5(\Omega\Sigma S^2\wedge S^3)\oplus \pi_5(\Omega\Sigma S^3\wedge S^2)\oplus \pi_5(\Omega\Sigma S^2\wedge S^2)\\
  &\cong \Z\langle (\iota_1\wedge \iota_2)_E\rangle\oplus \Z\langle (\iota_2\wedge \iota_1)_E\rangle\oplus\Z/2\langle (\iota_1\wedge \iota_1)_E\circ \eta_4\rangle,
\end{align*}

\begin{lemma}\label{lem:H_2-S3S4}
  Let $H_\sharp\colon \pi_5(\Omega\Sigma (S^2\vee S^3))\to \pi_5(\Omega\Sigma (S^2\vee S^3)^{\wedge 2})$ be the homomorphism induced by the second James-Hopf invariant. Assume that 
  \[H_\sharp(\Omega_0[\Sigma\iota_1,\Sigma\iota_2])=x\cdot (\iota_1\wedge \iota_2)_E+y\cdot (\iota_2\wedge \iota_1)_E+z\cdot  (\iota_1\wedge \iota_1)_E\circ \eta_4,\]
then $x,y$ are odd integers.

\begin{proof}
  Let $\sigma_n\in H_n(S^n)\cong\Z$ be a generator. There are isomorphisms
  \begin{align*}
    H_5(S^5)&\xra{\tau_{2,3}}H_5(S^2\wedge S^3)\cong H_2(S^2)\otimes H_3(S^3),\quad \sigma_5\mapsto (-1)^\alpha\cdot \sigma_2\otimes\sigma_3;\\
    H_5(S^5)&\xra{\tau_{3,2}}H_5(S^3\wedge S^2)\cong H_3(S^3)\otimes H_2(S^2),\quad \sigma_5\mapsto (-1)^\beta\cdot \sigma_3\otimes\sigma_2,
  \end{align*}
  where $\alpha,\beta$ are integers. By the Bott-Samelson theorem,
  \begin{align*}
    H_5(\Omega\Sigma (S^2\vee S^3))&\cong T\langle (\iota_1)_{E\ast}(\sigma_2), (\iota_2)_{E\ast}(\sigma_3)\rangle,\\
    H_5(\Omega\Sigma (S^2\vee S^3)^{\wedge 2})&\cong T\langle (\iota_i\wedge \iota_j)_{E\ast}(\sigma_{i+1}\otimes\sigma_{j+1})~|~i,j=1,2\rangle,
  \end{align*}
where $T\langle x_1,\cdots,x_m\rangle$ is the free tensor algebra $T(V)$ and $V$ is the free abelian group generated by $x_1,\cdots,x_m$.

Consider the following commutative diagram with $X=S^2\vee S^3$:
\begin{equation}\label{diag:H_2}
  \begin{aligned}
    \begin{tikzcd}[sep=scriptsize]
    \pi_5(\Omega\Sigma X)\ar[d,"H_\sharp"]\ar[r,"h"]&H_5(\Omega\Sigma X)\ar[r,"\rho_2"]\ar[d,"H_\ast"]&H_5(\Omega\Sigma X;\z{})\ar[d,"H_\ast"] \\
    \pi_5(\Omega\Sigma X^{\wedge 2})\ar[r,"h"]&H_5(\Omega\Sigma X^{\wedge 2})\ar[r,"\rho_2"]&H_5(\Omega\Sigma X^{\wedge 2};\z{})
  \end{tikzcd},
  \end{aligned}
\end{equation}
where $h$ are the Hurewicz homomorphisms, $\rho_2$ are the mod $2$ reductions.
Note that 
\[\Omega_0[\Sigma\iota_1,\Sigma\iota_2]=[\Omega_0\Sigma\iota_1,\Omega_0\Sigma\iota_2]^S=[(\iota_1)_E,(\iota_2)_E]^S.\]
 Then we compute that 
\begin{align*}
  h([(\iota_1)_E,(\iota_2)_E]^S)&=(-1)^{\alpha}[(\iota_1)_E,(\iota_2)_E]^S_\ast(\sigma_3\otimes\sigma_2)\\
  &=(-1)^{\alpha}\big((\iota_1)_{E\ast}(\sigma_3)\otimes (\iota_2)_{E\ast}(\sigma_2)-(\iota_2)_{E\ast}(\sigma_2)\otimes (\iota_1)_{E\ast}(\sigma_3)\big)\\
  &=(-1)^{\alpha}\big((\iota_1\wedge \iota_2)_{E\ast}(\sigma_2\otimes\sigma_3)-(\iota_2\wedge \iota_1)_{E\ast}(\sigma_3\otimes\sigma_2)\big), \\
  h((\iota_1\wedge \iota_2)_E)&=(-1)^\alpha (\iota_1\wedge \iota_2)_{E\ast}(\sigma_2\otimes \sigma_3),\\
  h((\iota_2\wedge \iota_1)_E)&=(-1)^{\beta} (\iota_2\wedge \iota_1)_{E\ast}(\sigma_3\otimes \sigma_2),\\
  h((\iota_1\wedge \iota_1)_E\circ \eta_4)&=0.
\end{align*}
For simplicity denote by $\bar{x}=\rho_2(x)$ for an element $x\in G$; write
\[\mathrm{RD}=H_\ast\rho_2h([(\iota_1)_E,(\iota_2)_E]^S),\quad \mathrm{DR}=\rho_2 h H_\sharp([(\iota_1)_E,(\iota_2)_E]^S).\] 
 \begin{align*}
  \mathrm{RD}&= H_\ast\big(\overline{(\iota_1)_{E\ast}(\sigma_3)}\otimes \overline{(\iota_2)_{E\ast}(\sigma_2)}-\overline{(\iota_2)_{E\ast}(\sigma_2)}\otimes \overline{(\iota_1)_{E\ast}(\sigma_3)}\big)\\
   &= \big(\overline{(\iota_1)_{E\ast}(\sigma_3)\otimes (\iota_2)_{E\ast}(\sigma_2)}-\overline{(\iota_2)_{E\ast}(\sigma_2)\otimes (\iota_1)_{E\ast}(\sigma_3)}\big)\\
   &=\big(\overline{(\iota_1\wedge \iota_2)_{E\ast}(\sigma_2\otimes\sigma_3)}-\overline{(\iota_2\wedge \iota_1)_{E\ast}(\sigma_3\otimes\sigma_2)}\big),\\[1ex]
   \mathrm{DR}&= \rho_2 \big( x\cdot h (\iota_1\wedge \iota_2)_E+y\cdot (\iota_2\wedge \iota_1)_E\big)\\
   &=\big(\bar{x}\cdot \overline{(\iota_1\wedge \iota_2)_{E\ast}(\sigma_2\otimes \sigma_3)}- \bar{y}\cdot \overline{(\iota_2\wedge \iota_1)_{E\ast}(\sigma_3\otimes\sigma_2)}\big).
 \end{align*}
Thus $\mathrm{RD}=  \mathrm{DR}$ implies $\bar{x}=\bar{y}=1$; i.e., $x,y$ are odd integers.
\end{proof}
\end{lemma}

Let $\Sigma\iota_1\colon \Sigma X\to \Sigma X\vee \Sigma Y$ and $\Sigma\iota_2\colon \Sigma Y\to \Sigma X\vee \Sigma Y$ be the canonical inclusions, respectively. 
Temporarily we define the \emph{weak second James-Hopf invariant} 
\[H^w(\Sigma X,\Sigma Y)\colon \pi_{n}(\Sigma X\wedge Y)\to \pi_{n}(\Sigma X\wedge Y)\] 
by the following defined commutative diagram 
\[\begin{tikzcd}[sep=scriptsize]
  \pi_n(\Sigma X\vee \Sigma Y)\ar[rr,"H_\sharp"]&&\pi_n(\Sigma (X\vee Y)^{\wedge 2})\ar[d,"p_{(\Sigma X,\Sigma Y)}"]\\
  \pi_{n}(\Sigma X\wedge Y)\ar[u,"{[\Sigma\iota_1,\Sigma\iota_2]_\sharp}"]\ar[rr,"{H^w(\Sigma X,\Sigma Y)}"]  &&\pi_n(\Sigma X\wedge Y)
\end{tikzcd},\] 
where $p_{(\Sigma X,\Sigma Y)}$ is the projection onto the second direct summand after composing with the natural isomorphism 
\[\pi_n(\Sigma(X\vee Y)^{\wedge 2})\cong \pi_n(\Sigma X\wedge X)\oplus \pi_n(\Sigma X\wedge Y)\oplus \pi_n(\Sigma Y\wedge X)\oplus\pi_n(\Sigma Y\wedge Y).\]

The two lemmas below follows by the naturality of the homomorphisms $H^w$ for appropriate spaces.
\begin{lemma}\label{lem:HW-nat}
  Let $\Sigma f\colon \Sigma X\to\Sigma X'$, $\Sigma g\colon Y\to \Sigma Y'$. If 
  \[(\Sigma f\wedge g)_\sharp\colon \pi_n(\Sigma X\wedge Y)\to \pi_n(\Sigma X'\wedge Y')\]
  is an isomorphism, then $H^w(\Sigma X,\Sigma Y)$ is an isomorphism if and only if $H^w(\Sigma X',\Sigma Y')$ is an isomorphism. 
\end{lemma}

\begin{lemma}\label{lem:HW-cof}
  Let $U\xra{f}V\xra{i}X\xra{q}\Sigma U$ be a homotopy cofibration which induces the short exact sequence 
  \[0\to \pi_n(\Sigma V\wedge Y)\xra{(\Sigma i\wedge 1_Y)_\sharp} \pi_n(\Sigma X\wedge Y)\xra{(\Sigma q\wedge 1_Y)_\sharp}\pi_n(\Sigma^2U\wedge Y)\to 0.\]
If both $H^w(\Sigma V,\Sigma Y)$ and $H^w(\Sigma^2U, \Sigma Y)$ are isomorphisms, then so is $H^w(\Sigma X,\Sigma Y)$.
\end{lemma}

\begin{lemma}\label{lem:htpgrps-smash}
 $\pi_6(P^3(2^r)\wedge C^5_\eta)=\pi_6(C^4_\eta\wedge C^5_\eta)=\pi_6(C^4_\eta\wedge C^5_r)=0$.
 \begin{proof}
  By \cite[Lemma 4.2]{ZP17}, $\pi_6(P^3(2^r)\wedge C^5_\eta)=\pi_7(P^4(2^r)\wedge C^5_\eta)=0$.
  By the same method used in the proof of \cite[Lemma 4.2]{ZP17}, we have 
  \[sk_8(C^4_\eta\wedge C^5_\eta)\simeq S^7\vee C^7_\eta,\quad sk_8(C^4\wedge C^5_r)\simeq S^7\vee P^3(2^r)\wedge C^5_\eta,\]
  where $sk_8(X)$ denotes the $8$-dimensional skeleton of $X$. Thus 
  \[\pi_6(C^4\wedge C^5_\eta)\cong \pi_6(C^7_\eta)=0,\quad \pi_6(C^4_\eta\wedge C^5_r)\cong \pi_6(P^3(2^r)\wedge C^5_\eta)=0.\]
 \end{proof}
\end{lemma}

\begin{lemma}\label{lem:HW}
  The weak second James-Hopf invariant 
  \[H^w(\Sigma X,\Sigma Y)\colon \pi_6(\Sigma X\wedge Y)\to \pi_6(\Sigma X\wedge Y)\]
is an isomorphism for $(X,Y)=(S^2,S^2), (S^2,P^3(2^s))$, $(S^3,P^3(2^s))$, $(S^3,C^4(2^s)) $, $(P^3(2^r),P^3(2^s))$, $(P^3(2^r),C^4_s)$ and $(C^4_r,C^4_s)$.

  \begin{proof}
Denote by $1_n$ the identity of $S^n$.
(1) Observe that  
    \[(\Omega\Sigma \eta_4)\circ E=(\Omega\Sigma 1_2\wedge \eta)\circ E=\Omega_0(\Sigma 1_2\wedge \eta)\] is a generator of $\pi_5(\Omega\Sigma S^2\wedge S^2)$.   Consider the following commutative diagram 
    \[\begin{tikzcd}
      \pi_5(\Omega\Sigma S^2\wedge S^3)\ar[d,"(\Omega \Sigma 1_2\wedge \eta)_\sharp"]\ar[rr,"{H^w(S^3,S^4)}"]&&\pi_5(\Omega\Sigma S^2\wedge S^3)\ar[d,"(\Omega \Sigma 1_2\wedge \eta)_\sharp"]\\
      \pi_5(\Omega\Sigma S^2\wedge S^2)\ar[rr,"{H^w(S^3,S^3)}"]&&\pi_5(\Omega\Sigma S^2\wedge S^3)
    \end{tikzcd}.\]
Then by Lemma \ref{lem:H_2-S3S4} we have 
\begin{align*}
  H^w(S^3,S^3)\big((\Omega\Sigma 1_2\wedge \eta)\circ E\big)&=(\Omega \Sigma 1_2\wedge \eta) H^w(S^3,S^4)(E)\\
  &=(\Omega \Sigma 1_2\wedge \eta)p_{(S^3,S^4)}H_\sharp (\Omega_0[\Sigma \iota_1,\Sigma\iota_2])\\
  &=(\Omega \Sigma 1_2\wedge \eta)p_{(S^3,S^4)} (x\cdot (\iota_1\wedge \iota_2)_E)\\
  &=(\Omega \Sigma 1_2\wedge \eta)(x\cdot E)\\
  &=(\Omega \Sigma 1_2\wedge \eta)\circ E.
\end{align*}
Thus $H^w(S^3,S^3)$ is an isomorphism.

(2) By Lemma \ref{lem:H_2-S3S4}, $H^w(S^4,S^3)$ is surjective. There is a commutative diagram 
\[\begin{tikzcd}
  \pi_6(\Sigma S^3\wedge S^2)\ar[d,"(\Sigma 1_3\wedge i_2)_\sharp",two heads]\ar[rr,"{H^w(S^4,S^3)}",two heads]&& \pi_6(\Sigma S^3\wedge S^2)\ar[d,"(\Sigma 1_3\wedge i_2)_\sharp",two heads]\\
  \pi_6(\Sigma S^3\wedge P^3(2^s))\ar[rr,"{H^w(S^4,P^3(2^s))}"]&&\pi_6(\Sigma S^3\wedge P^3(2^s))
\end{tikzcd},\]
where ``$\twoheadrightarrow $'' denote surjections. It follows that $H^w(S^3,P^3(2^s))$ is an epimorphism of $\z{r}$, which must be an isomorphism. 

(3) There is an isomorphism 
\[(\Sigma 1_2\wedge i_2)_\sharp\colon \pi_6(\Sigma S^2\wedge S^2)\to \pi_6(\Sigma S^2\wedge P^3(2^s)).\]
Then $H^w(S^4,P^4(2^s))$ is an isomorphism by Lemma \ref{lem:HW-nat}.

(4)  Recall from \cite{Baues85} or \cite{lipc22} $\pi_n(C^{n+2}_r)\cong\z{r}$ and $\pi_{n+1}(C^{n+2}_r)=0$ for $n\geq 3$.
Hence we have an isomorphism 
\[(1_4\wedge i_P)_\sharp\colon \pi_6(S^4\wedge P^3(2^s))\to \pi_6(S^4\wedge C^4_s),\] 
which implies that $H^w(S^4,C^5_s)$ is an isomorphism.

(5) There is a short exact sequence 
\begin{multline*}
  0\to \pi_6(\Sigma S^2\wedge P^3(2^s))\xra{(\Sigma i_2\wedge 1_P)_\sharp}\pi_6(\Sigma P^3(2^r)\wedge P^3(2^s))\\\xra{(\Sigma q_3\wedge 1_P)_\sharp}\pi_6(\Sigma S^3\wedge P^3(2^s))\to 0.
\end{multline*}
It follows by former conclusions and Lemma \ref{lem:HW-cof} that $H^w(P^4(2^r),P^4(2^s))$ is an isomorphism. 

(6) We may assume that $r\geq s$. Then the middle homomorphism $(q_4\wedge id)_\sharp$ in the following exact sequence is an isomorphism:
\[\begin{tikzcd}
  \pi_6(S^3\wedge C^4_s)\ar[r,"(i_3\wedge id)_\sharp"]&\pi_6(P^4(2^r)\wedge C^4_s)\ar[r,"(q_4\wedge id)_\sharp"]&\pi_6(S^4\wedge C^4_s)\ar[r,"(2^r 1_4\wedge id)_\sharp"]& \pi_6(S^4\wedge C^4_s)
\end{tikzcd}\] 
It follows that $H^w(P^4(2^r),C^5_s)$ is an isomorphism, by Lemma \ref{lem:HW-nat}.

(7) By Lemma \ref{lem:htpgrps-smash}, the cofibration 
\(S^3\xra{i_32^r}C^5_\eta\xra{i_\eta}C^5_r\xra{q_4}S^4\)
induces the exact sequence 
\[0\to \pi_6(C^5_r\wedge C^4_s)\xra{(q_4\wedge id)_\sharp}\pi_6(S^4\wedge C^4_s)\xra{(2^ri_4\wedge id)_\sharp=0}\pi_6(C^6_\eta\wedge C^4_s).\]
Thus $(q_4\wedge id)_\sharp\colon \pi_6(C^5_r\wedge C^4_s)\to \pi_6(S^4\wedge C^4_s)$ is an isomorphism, and therefore $H^w(C^5_r,C^5_s)$ is an isomorphism, by Lemma \ref{lem:HW-nat}.
  \end{proof}
\end{lemma}

\bibliographystyle{amsplain}
\bibliography{refs}

\end{document}